\documentclass[a4paper]{tran-l}

\usepackage{amssymb}
\usepackage{hyperref}
\usepackage{enumitem}
\usepackage{tikz-cd}
\usepackage{mathrsfs}
\usepackage{relsize}
\usepackage[bbgreekl]{mathbbol}
\usepackage{amsfonts}
\DeclareSymbolFontAlphabet{\mathbb}{AMSb} 
\DeclareSymbolFontAlphabet{\mathbbl}{bbold}
\newcommand{\Prism}{{\mathlarger{\mathbbl{\Delta}}}} 

\newtheorem{theorem}{Theorem}[section]
\newtheorem{lemma}[theorem]{Lemma}
\newtheorem{prop}[theorem]{Proposition}
\newtheorem{conjecture}[theorem]{Conjecture}
\newtheorem{corollary}[theorem]{Corollary}

\theoremstyle{definition}
\newtheorem{definition}[theorem]{Definition}
\newtheorem{example}[theorem]{Example}

\newtheorem{convention}[theorem]{Convention}
\newtheorem{assumption}[theorem]{Assumption}

\theoremstyle{remark}
\newtheorem{remark}[theorem]{Remark}

\numberwithin{equation}{section}

\def\A              {\mathbb{A}}

\def\cA             {\mathcal{A}}

\def\fA             {\mathfrak{A}}
\def\cB             {\mathcal{B}}
\def\C              {\mathbb{C}}
\def\D              {\mathcal{D}}
\def\ccE            {\mathscr{E}}
\def\F              {\mathbb{F}}
\def\cF             {\mathcal{F}}
\def\G              {\mathbb{G}}
\def\cG             {\mathcal{G}}
\def\cH             {\mathcal{H}}
\def\cI             {\mathcal{I}}
\def\Q              {\mathbb{Q}}

\def\N              {\mathbb{N}}
\def\P              {\mathfrak{P}}

\def\R              {\mathbb{R}}
\def\S              {\mathbb{S}}
\def\cS             {\mathcal{S}}

\def\Z              {\mathbb{Z}}
\def\O              {\mathcal{O}}
\def\M              {\mathcal{M}}
\def\Y              {\mathcal{Y}}
\def\X              {\mathfrak{X}}
\def\cX             {\mathcal{X}}

\def\iso            {\cong}

\def\g              {\mathfrak{g}}

\def\Aut            {\operatorname{Aut}}
\def\alg            {\operatorname{alg}}
\def\Hom            {\operatorname{Hom}}
\def\ad             {\operatorname{ad}}

\def\Bun            {\operatorname{Bun}}
\def\coeq           {\operatorname{coeq}}

\def\cris           {\operatorname{cris}}
\def\crys           {\operatorname{crys}}
\def\cycl           {\operatorname{cycl}}
\def\der            {\operatorname{der}}
\def\det            {\operatorname{det}}

\def\dR             {\operatorname{dR}}
\def\End            {\operatorname{End}}
\def\Ext            {\operatorname{Ext}}
\def\Fl             {\mathcal{F}l}
\def\Frac           {\operatorname{Frac}}
\def\Gal            {\operatorname{Gal}}
\def\GL             {\operatorname{GL}}
\def\Gr             {\operatorname{Gr}}
\def\Grpd           {\operatorname{Grpd}}
\def\GSp            {\operatorname{GSp}}
\def\id             {\operatorname{id}}

\def\Igs            {\operatorname{Igs}}
\def\Isom           {\operatorname{Isom}}
\def\ker            {\operatorname{ker}}
\def\Res            {\operatorname{Res}}
\def\Lie            {\operatorname{Lie}}

\def\Nilp           {\operatorname{Nilp}}
\def\Nrd            {\operatorname{Nrd}}
\def\opp            {\operatorname{op}}
\def\Perf           {\operatorname{Perf}}
\def\Perfd          {\operatorname{Perfd}}

\def\pre            {\operatorname{pre}}
\def\Proj           {\operatorname{Proj}}
\def\red            {\operatorname{red}}
\def\Rep            {\operatorname{Rep}}
\def\Sh             {\operatorname{Sh}}

\def\Spa            {\operatorname{Spa}}
\def\Spd            {\operatorname{Spd}}
\def\Spf            {\operatorname{Spf}}
\def\Spec           {\operatorname{Spec}}

\def\Tr             {\operatorname{Tr}}

\def\invlim         {\varprojlim}
\def\dirlim         {\varinjlim}

\begin{document}

\title{Igusa stacks for certain abelian-type Shimura varieties}
\author{Fabian Schnelle}
\maketitle

\begin{abstract}
    We construct Igusa stacks for the good reduction locus of a class of abelian-type Shimura varieties that can be defined in terms of a PEL datum, under the assumption that it is of type (A even) or (C) and unramified at a prime $p$.
\end{abstract}

\tableofcontents
\section{Introduction}
In this paper, we are concerned with Scholze's fiber product conjecture for Shimura varieties with infinite level at $p$, as it was first formulated in Mingjia Zhang's PhD thesis \cite{Zha23}. In rough terms, it is a conjecture about how the $p$-adic geometry of Shimura varieties with infinite level at $p$ can be separated into a $p$-part, which is to be modelled by the Schubert cell in the mixed-characteristic affine Grassmannian labelled by the inverse of the Hodge cocharacter of the Shimura datum, and a prime-to-$p$ part, which is to be modelled by a v-stack which is supposed to interpolate between Igusa varieties. For this reason, Zhang called these stacks ``Igusa stacks".

\subsection{Igusa stacks}
To state the conjecture, consider the following setup. Let $(G,X)$ be a Shimura datum with associated $G(\C)$-conjugacy class of minuscule cocharacters $[\mu^{-1}]$ and field of definition $E_0$. We fix a prime $p$ and an isomorphism $\C \iso \overline{\Q}_p$, and complete $E_0$ with respect to the prime above $p$ that is induced by the map $E_0 \to \C \iso \overline{\Q}_p$. We will denote the resulting local field by $E$, and its residue field by $\F_q$. Consider a compact open subgroup $K = K_p K^p \subseteq G(\A_f)$ and the category $\Perf$ of perfectoid spaces in characteristic $p$, which we endow with the v-topology. By $\cS_K$ we will denote the diamond over $\Spd E$ attached to the adic analytification of the Shimura variety at level $K$, and by $\cS_{K^p} = \invlim_{K^p} \cS_{K_p K^p}$ the Shimura variety with infinite level at $p$. We also let $\cS$ be the Shimura variety at infinite level everywhere, again considered as a diamond. Furthermore, let $\Gr_G$ be the $B_{\dR}^+$-affine Grassmannian over $\Spd E$ attached to $G_{\Q_p}$. We will also fix a maximal torus inside a Borel subgroup of $G_{\overline{\Q}_p}$, and a dominant cocharacter $\mu \in X_*^+(T)$ representing the $G(\overline{\Q}_p)$-conjugacy class of $[\mu]$. Denote by $\Gr_{G,\mu}$ the Schubert cell labelled by $\mu$. Furthermore let $\Bun_G$ be the small v-stack on $\Perf$ that parametrizes $G_{\Q_p}$-bundles on the Fargues-Fontaine curve. Recall the Beauville-Laszlo uniformization map $BL: \Gr_G \to \Bun_G$ and the Hodge-Tate period map $\pi_{HT}: \cS_{K^p} \to \Gr_G$, whose image lies in $\Gr_{G,\mu}$. 

With this setup we can now state the (rational part of the) fiber product conjecture:
\begin{conjecture}[Scholze, cf. \cite{Zha23}]\label{conjecture}
    There exists a construction of a system of small v-stacks $\{\Igs_{K^p}\}_{K^p}$ on $\Perf$, together with maps $\red: \cS_{K^p}/\Spd E \to \Igs_{K^p}$ and $\overline{\pi}_{HT}: \Igs_{K^p} \to \Bun_G$, such that
    \begin{enumerate}[label=(\roman*)]
        \item For each $K^p$, the following diagram is Cartesian:
\[\begin{tikzcd}
	{\cS_{K^p}} & {\Gr_{G,\mu}} \\
	{\Igs_{K^p}} & {\Bun_G}
	\arrow["{\pi_{HT}}", from=1-1, to=1-2]
	\arrow["\red"', from=1-1, to=2-1]
	\arrow["BL", from=1-2, to=2-2]
	\arrow["{\overline{\pi}_{HT}}"', from=2-1, to=2-2]
\end{tikzcd}\]
    \item There exists a $G(\A_f)$-action on the system  $\{\Igs_{K^p}\}_{K^p}$ of Igusa stacks (with $G(\Q_p)$ acting trivially) that descends the Hecke action on $\{\cS_{K^p}\}_{K^p}$. 
    \end{enumerate}
\end{conjecture}

On a historical note, Zhang proved the conjecture for PEL-type Shimura data of type (A) and (C) in the classification of Kottwitz in her PhD thesis \cite{Zha23}, under the assumptions that the reductive group $G$ is unramified at $p$, and also with the imperfection that the Igusa stacks she defined were defined over $\Spd\F_q$ instead of $\Spd\F_p$. Later Patrick Daniels, Pol van Hoften, Dongryul Kim and Mingjia Zhang generalized the result to arbitrary Shimura data of Hodge type and also managed to lower the base to $\Spd\F_p$, see \cite{DHKZ24}. In both cases they proved the result only on the good reduction locus of the Shimura varieties in question (which agrees with the whole Shimura variety if the latter is compact).

It is usually more convenient to state Conjecture \ref{conjecture} after passing to infinite level away from $p$ as well. This can be done because the Hodge-Tate period map $\pi_{HT}$ (and also $\overline{\pi}_{HT}$, once constructed) is invariant under changing the away-from-$p$ level. At least in the Hodge-type setup, one of the benefits of working with the Igusa stack at infinite level is that it turns out to be a v-sheaf (see \cite[5.2.6]{DHKZ24}).

In the recent preprint \cite{Kim25}, Dongryul Kim axiomatized the notion of Igusa stacks as follows.
\begin{definition}[\cite{Kim25}, 5.5]\label{igs_axioms}
    An \textbf{Igusa stack} is a v-stack $\Igs(G,X)$ equipped with a $\underline{G(\A_f^p)}$-action and a Cartesian diagram
\[\begin{tikzcd}
	\cS & {\Gr_{G,\mu}} \\
	\Igs & {\Bun_{G,\mu}}
	\arrow["{\pi_{HT}}", from=1-1, to=1-2]
	\arrow[from=1-1, to=2-1]
	\arrow["BL", from=1-2, to=2-2]
	\arrow["{\overline{\pi}_{HT}}", from=2-1, to=2-2]
\end{tikzcd}\]
satisfying the following two axioms
\begin{itemize}
    \item The $G(\A_f^p) \times G(\Q_p)$ action on $\Igs \times_{\Bun_{G,\mu}} \Gr_{G,\mu}$ recovers the Hecke action of $G(\A_f)$ on $\cS$.
    \item For $\phi$ the absolute Frobenius, $\phi \times \id$ acts trivially on $\Igs \times_{\Bun_{G,\mu}} \Gr_{G,\mu}.$
\end{itemize}
\end{definition}
With this axiomatization, Kim showed that the construction of Igusa stacks (either on the full Shimura variety or its good reduction locus), if they exist, is automatically functorial in Shimura data (see \cite[Theorem B]{Kim25}, and furthermore that Igusa stacks are unique up to unique isomorphism, again if they exist (see \cite[Theorem A]{Kim25}). Both uniqueness and functoriality can be proven from the first axiom alone, see \cite[Remark 5.7]{Kim25} - the second axiom only comes into play once we actually want to compare the cohomology of Igusa stacks to the cohomology of Shimura varieties.

\subsection{Goal of this paper} 
In this paper, we construct Igusa stacks for the good reduction locus of certain Shimura varieties of abelian type which arise naturally from a PEL datum (see the appendix \cite[A3]{Nek18}). Since they are of abelian type, they are not covered by the results of \cite{DHKZ24}. Examples of such Shimura varieties include Hilbert modular varieties or totally indefinite quaternionic Shimura varieties, but in principle any PEL datum of type (A even) or (C) that is unramified at $p$ will yield such abelian-type Shimura varieties.

What makes strictly abelian-type Shimura varieties difficult to handle is that they do not possess a fine moduli description in terms of abelian varieties with additional structures, and also cannot be embedded into Siegel modular varieties. Thus at first glance, all methods from \cite{Zha23} or \cite{DHKZ24} to construct Igusa stacks break down. 

The fact that the abelian-type Shimura varieties that we are considering stem from a PEL datum makes them substantially easier to handle. In fact, they arise in a two-step process from the PEL-type Shimura varieties attached to the same PEL datum.

In the first step, one enlarges the PEL-type Shimura variety by allowing more general level structures in the moduli problem, which we will call \textbf{hybrid level structures}; this again yields fine moduli spaces, which we will call \textbf{hybrid spaces}. These are still very well behaved, and basically consist of multiple copies of the PEL-type Shimura variety (see Definition \ref{hybrid_spaces}). In the second step, one quotients out a certain polarization action on these hybrid spaces, which will then yield the desired abelian-type Shimura varieties. 

Our main result is as follows.
\begin{theorem}[Corollary \ref{cart_abelian}]\label{thm_introduction}
    There is an Igusa stack $\Igs$ for the abelian-type Shimura datum $(G,X)$ described above, satisfying the axioms of Definition \ref{igs_axioms}.
\end{theorem}
We refer to Chapter \ref{igs_chapter} for the explicit construction of $\Igs$.

We remark that Patrick Daniels, Pol van Hoften, Dongryul Kim and Mingjia Zhang have already announced results that will construct Igusa stacks for all abelian-type Shimura varieties. Their methods will heavily lean on the methods developed in \cite{Kim25}, and thus can be expected to be less concrete. On the other hand, by restricting our attention to this special class of abelian-type Shimura varieties, we are able to give a rather concrete construction of Igusa stacks, much in the spirit of \cite{Zha23}.

\subsection{Strategy of proof}
We will make heavy use of the fact that our abelian-type Shimura variety arises in the two-step process described above from the PEL-type Shimura variety.

Namely, we first tackle the situation for the hybrid space at infinite level everywhere. In this step, it is rather straightforward to use analogous methods as in \cite{Zha23} and \cite{DHKZ24} to construct a hybrid Igusa stack fitting into a similar Cartesian diagram.

In a second step, we note that the hybrid Igusa stack (which turns out to be a v-sheaf again) comes equipped with a similar polarization action as the hybrid space. A natural idea is to naively take the stacky quotient by this quotient, but this approach immedaitely runs into problems. This is due to the fact that the polarization action on the hybrid spaces at finite levels is almost never free (in particular not for principal level-$N$ structures, even for $N$ sufficiently large). However, by a careful analysis we manage to find a cofinal system of level subgroups so that the polarization action at these levels is indeed free. Thus quotienting out the polarization action at these levels yields finite étale torsors. After passing to infinite level everywhere, this shows that quotienting out the polarization action on the hybrid space at infinite level and on the hybrid Igusa stack at infinite level yields pro-étale torsors with respect to the same profinite group.

Thus after taking the naive stacky quotient of the hybrid Igusa stack at infinite level, the fiber product diagram for the abelian-type Shimura variety is an immediate consequence of the fiber product diagram for the hybrid space. 

\subsection{Acknowledgements}
This work grew out of the author's master thesis at the University of Bonn, in which only the special case of Hilbert modular varieties was considered. First and foremost, we would like to heartily thank Siyan Daniel Li-Huerta for suggesting the problem of constructing Igusa stacks for Hilbert modular varieties, and for his very close and detailed supervision throughout the whole project. Furthermore, we would like to thank Peter Scholze for agreeing to supervise the thesis; in particular, he made the suggestion to go to infinite level everywhere before taking quotients. 

During the process of generalizing the results of the thesis, the author was supported by Germany’s Excellence Strategy EXC 2044–390685587 “Mathematics Münster: Dynamics–Geometry–Structure”, by the CRC 1442 “Geometry: Deformations and Rigidity” of the DFG, and by the Leibniz prize of Eva Viehmann.

\section{Shimura varieties and hybrid spaces}\label{section_sv}
In this section we first recall the definition of Shimura data and Shimura varieties for the reader's convenience. Then, we discuss the PEL datum that gives rise to the abelian-type Shimura varieties which are our key concern in this paper. Along the way, we also discuss the aforementioned hybrid spaces, their moduli descriptions and integral models. 

\subsection{Shimura data and Shimura varieties}
Let $\S := \Res_{\C/\R}\G_m$ be the Deligne torus.

\begin{definition}[\cite{Del79}, 2.1.1]
    A \textbf{Shimura datum} is a pair $(G,X)$ consisting of a reductive group $G$ over $\Q$ and a $G(\R)$-conjugacy class $X$ of homomorphisms $h: \S \to G_\R$ satisfying the following three axioms:
    \begin{itemize}
        \item[(SV1)] For all $h \in X$, only the weights $(1,-1), (0,0)$ and $(-1,1)$ may appear in the adjoint representation of $h_\C$ on the complexified Lie algebra $\g_\C$ of $G$.
        \item[(SV2)] For all $h \in X$, the adjoint action of $h(i)$ on $G_{\R}^{\ad}$ is a Cartan involution.
        \item[(SV3)] $G^{\ad}$ does not admit any factor $H$ over $\Q$ such that the projection of $h$ onto $H$ is trivial.
    \end{itemize}
\end{definition}

As explained in \cite{Del79}, axiom (SV2) ensures that the set $X$ naturally acquires the structure of a complex manifold, with the stabilizer of every $h \in X$ being compact modulo center. Axiom (SV3) guarantees that the simply connected cover of the derived subgroup $G^{\der}$ satisfies strong approximation.

\begin{definition}
    A \textbf{morphism between Shimura data} $(G,X)$ and $(G',X')$ is a homomorphism $G \to G'$ of algebraic groups over $\Q$ that maps $X$ to $X'$.
\end{definition}

Given a Shimura datum $(G,X)$ and a compact open subgroup $K \subset G(\A_f)$, consider the double coset space
\begin{align*}
    X_K := G(\Q)\backslash(X \times G(\A_f))/K,
\end{align*}
where $G(\Q)$ acts diagonally from the left on both factors (using the obvious embeddings $G(\Q) \hookrightarrow G(\R)$ and $G(\Q) \hookrightarrow G(\A_f)$) and $K$ acts on $G(\A_f)$ by multiplication from the right. If $K$ is sufficiently small, $X_K$ is a complex manifold (\cite{Mil17}, 3.1, 5.13). By the theorems of Baily-Borel and Borel, there is a unique quasi-projective variety $\Sh_K(G,X)_\C$ over $\C$ whose complex analytification is $X_K$.

\begin{definition}
    The \textbf{Hecke action} of $G(\A_f)$ on the inverse system $\{\Sh_K(G,X)_\C\}_K$ via isomorphisms is given by
    \begin{align*}
        \gamma_g: \Sh_K(G,X)_\C \to \Sh_{g^{-1}Kg}(G,X)_\C
    \end{align*}
    where $\gamma_g$ maps a double coset $[h,g_1]$ to $[h,g_1g]$.
\end{definition}

To a Shimura datum $(G,X)$ and $h \in X$, we associate the \textbf{Hodge cocharacter} $\nu_h: \G_{m,\C} \to G_\C$, defined as
\begin{align*}
    \nu_h(z) := h_\C(z,1).
\end{align*}
Denote by $[\nu_h]$ its $G(\C)$-conjugacy class, which does not depend on the choice of $h \in X$. In the introduction, this conjugacy class was denoted by $[\mu^{-1}]$. The automorphism group $\Aut(\C/\Q)$ acts on the set of $G(\C)$-conjugacy classes of cocharacters $\G_{m,\C} \to G_\C$. By \cite[12.1]{Mil17}, $\nu_h$ has a conjugate that is already defined over $\overline{\Q}$, thus we may also regard $[\nu_h]$ as a $G(\overline{\Q})$-conjugacy class.

\begin{definition}
    Let $(G,X)$ be a Shimura datum. Its associated \textbf{reflex field} or \textbf{field of definition} $E(G,X)$ is the fixed field of the stabilizer of $[\nu_h]$ in $\Aut(\C/\Q)$. 
\end{definition}

\begin{theorem}[Deligne, Milne, Borovoi]
    Let $(G,X)$ be a Shimura datum. The system $\{\Sh_K(G,X)_\C\}_K$, equipped with Hecke action by $G(\A_f)$, has a unique model $\{\Sh_K(G,X)\}_K$ over $E(G,X)$, subject to the following requirements.
    \begin{itemize}
        \item The Hecke action is defined over $E(G,X)$.
        \item All special points (i.e. points $[h,g] \in X_K$ such that $h$ factors through a $\Q$-subtorus of $G$) are algebraic.
        \item The Galois action on the special points is normalized as in \cite[2.2.9]{Del79}.
    \end{itemize}
\end{theorem}

\begin{definition}
    Let $(G,X)$ be a Shimura datum. For a sufficiently small compact open subgroup $K \subset G(\A_f)$, we call $\Sh_K(G,X)$ the \textbf{Shimura variety} at level $K$. The inverse limit $\lim\limits_{\leftarrow K}\Sh_K(G,X)$, which is a scheme over $E(G,X)$, is called the \textbf{(canonical model of the) Shimura variety attached to the Shimura datum} $(G,X)$ and is denoted by $\Sh(G,X)$.
\end{definition}
For compact open subgroups $K' \subset K$, the morphism $\Sh_K(G,X) \to \Sh_{K'}(G,X)$ is finite étale. Hence, the inverse limit 
\[\Sh(G,X) :=\invlim_{K} \Sh_K(G,X)\]
exists as a scheme that is regular and locally (!) of finite type. This scheme is simply called the \textbf{Shimura variety} attached to the Shimura datum $(G,X)$.

\subsection{Small linear algebraic interlude}
The following basic lemma, which is about switching between non-degenerate bilinear pairings along finite étale ring maps, will come in handy throughout.

\begin{lemma}\label{forms_trace}
    Let $R \to A$ be a finite étale map and let $W$ be a finite locally free $A$-module. Then given a non-degenerate $R$-bilinear form $(\cdot,\cdot): W \times W \to R$ that satisfies $(aw,w') = (w,aw')$ for all $w,w' \in W$ and $a \in A$, one can lift this pairing uniquely to a non-degenerate $A$-bilinear pairing $\langle\cdot,\cdot\rangle: W \times W \to A$  that satisfies
    \[\Tr_{A/R} \circ \langle\cdot,\cdot\rangle = (\cdot,\cdot)\]
\end{lemma}
\begin{proof}
    The trace pairing is perfect since $R \to A$ is finite étale, thus to define $\langle w,w'\rangle \in A$ for any $w,w' \in W$ it suffices to specify that for all $a \in A$
    \[\Tr_{A/R}(a \langle w,w'\rangle) = (aw,w').\]
    One immediately checks that the bilinear form thus defined is indeed non-degenerate and $A$-linear. 

    Assume we have two lifts $\langle\cdot,\cdot\rangle$ and $\langle\cdot,\cdot\rangle'$, then for all $a \in A$
    \[\Tr_{A/R}(a\langle w,w'\rangle) = (aw,w') = \Tr_{A/R}(a\langle w,w'\rangle'),\]
    which by perfectness of the trace pairing implies $\langle\cdot,\cdot\rangle = \langle\cdot,\cdot\rangle'$.
\end{proof}

\begin{convention}
    In the rest of this paper we will stick to the following convention: Whenever we are in the setting of Lemma \ref{forms_trace}, angular brackets $\langle\cdot,\cdot\rangle$ will denote pairings on $A$, and composition with the trace will be denoted by round brackets $(\cdot,\cdot) :=\Tr_{A/R} \circ \langle\cdot,\cdot\rangle$.
\end{convention}

The interplay with symplectic similitudes is as follows. Assume $f: W \to W'$ is an $A$-linear morphism between finite locally free $A$-modules that are equipped with non-degenerate $R$-bilinear pairings $(\cdot,\cdot)_W$ resp. $(\cdot,\cdot)_{W'}$ satisfying the assumption of the above lemma. Then the following are equivalent:
\begin{enumerate}
    \item $f$ is an $A$-symplectic similitude with respect to the lifted pairings, i.e. there is some $a \in A^\times$ such that for all $w,w' \in W$
        \[\langle f(w),f(w')\rangle_{W'} = a \langle w,w'\rangle_W.\]
    \item There exists an $a \in A^\times$ such that for all $w,w' \in W$
        \[(w,w')_W = \Tr_{A/R}(a \langle f(w),f(w')\rangle_{W'}).\]
    \item There exists an $a \in A^\times$ such that 
        \[(f(w),f(w'))_{W'} = (aw,w')_W = (w,aw')_W\]
\end{enumerate}

\subsection{The PEL data}
From now on, we fix the following PEL data as in \cite[A3]{Nek18}. 

\subsubsection{Global side}
We fix a finite-dimensional simple $\Q$-algebra $B$, equipped with a $\Q$-linear positive involution $(-)^*$. In this case $F := Z(B)^{* = \id}$ is a totally real number field of degree $g = [F:\Q]$. We let $V$ be a non-zero finite-type left $B$-module, equipped with a non-degenerate alternating $F$-bilinear form $\langle\cdot,\cdot\rangle_F$ that satisfies $\langle bv,v'\rangle_F = \langle v,b^*v'\rangle_F$ for all $v,v' \in V$ and $b \in B$. Composing it with the trace $\Tr_{F/\Q}$ yields the non-degenerate alternating $\Q$-bilinear pairing
\[(\cdot,\cdot)_F := \Tr_{F/\Q} \circ \langle\cdot,\cdot\rangle_F\]
that satisfies the same Hermitian property as $\langle\cdot,\cdot\rangle_F$.

The center $F_c := Z(B)$ is either equal to the totally real field $F$ or it is a CM field with maximal totally real subfield $F$ and involution $(-)^*$. Then $C = \End_B(V)$ is a simple $\Q$-algebra with center $F_c$ and an $F$-linear involution $(-)^\#$ that is given by the adjoint with respect to $\langle\cdot,\cdot\rangle_F$.

Let $H := \GSp_{B}(V,\langle\cdot,\cdot\rangle)$ be the algebraic group over $F$ whose points in any $F$-algebra $S$ are given by
\[H(S) = \{h \in \GL_B(V \otimes_F S) : \exists\nu(h) \in S^ \times \text{ } \forall v,v' \in V \text{ } \langle hv,hv'\rangle_F = \nu(h)\langle v,v'\rangle_F\}.\]
The scalars $\nu(h) \in S^\times$ assemble to a morphism of algebraic groups $\nu: H \to \G_{m,F}$. Via Weil restriction, we obtain the algebraic group 
\[G := \Res_{F/\Q}(H),\]
together with a morphism $G \to \Res_{F/\Q}(\G_m)$ which we will again denote by $\nu$. Let $G^*$ be the algebraic group over $\Q$ defined by the pullback diagram

\[\begin{tikzcd}
	{G^*} & G \\
	{\G_m} & {\Res_{F/\Q}(\G_m)}
	\arrow[hook, from=1-1, to=1-2]
	\arrow["\nu", from=1-1, to=2-1]
	\arrow["\ulcorner"{anchor=center, pos=0.125}, draw=none, from=1-1, to=2-2]
	\arrow["\nu", from=1-2, to=2-2]
	\arrow[hook, from=2-1, to=2-2]
\end{tikzcd}\]
Alternatively, we can define $H$ as the functor of points
\[H(S) = \{h \in (C \otimes_F S)^\times : hh^\# = \nu(h) \in S^\times\}\]
for any $F$-algebra $S$, in which case we deduce for any $\Q$-algebra $R$ that
\[G^*(R) = \{g \in (C \otimes R)^\times : gg^\# = \nu(g) \in R^\times\}.\]

We quickly recall the classification of the pair $(B,*)$ from \cite[§1]{Kot92}. It is said to be of type (A) if the pair $(B \otimes_F \overline{F}, * \otimes \id)$ is isomorpic to $\End(W) \times \End(W)^{\opp}$ with $(a,b)^* = (b,a)$, and of type (C) if $(B \otimes_F \overline{F}, * \otimes \id)$ is isomorphic to $\End(W)$ with $(-)^*$ being the adjoint map with respect to a symmetric bilinear form on $W$. If $(-)^*$ is the adjoint map with respect to an alternating bilinear form on $W$, we are in type (B) or (D).

If we are in type (C), then 
\[G_\R = \prod_{v|\infty} \GSp(2n)_\R\]
for some $n \geq 1$, and if we are in type (A), then
\[G_\R = \prod_{v |\infty} GU(a_v,b_v)\]
with $a_v + b_v = n$ for some $n \geq 1$. We will say that we are in type (A even) if $n$ is even, and in type (A odd) if $n$ is odd. 

\begin{assumption}
    From now on, we will always assume that the data we fixed is either of type (A even) (which implies $F_c \neq F$) or of type (C) (which implies $F_c = F$.
\end{assumption}

Under this assumption, the algebraic groups $H$, $G$ and $G^*$ are connected reductive groups, and furthermore the derived subgroups of $H$, and thus of $G$, are simply connected (see \cite[8.7]{Mil17}). 

There exists a morphism of $\R$-algebras $h: \C \to C \otimes \R$ such that $h(\overline{z})=h(z)^\#$ and such that the symmetric $\R$-bilinear form $\langle v,h(i)v'\rangle: V_\R \times V_\R \to \R$ is positive definite. Up to conjugation by an element $c \in (C \otimes \R)^\times$ such that $cc^\# = 1$, this morphism is unique (see \cite[4.3]{Kot92} and \cite[8.12]{Mil17}).

Thus $h$ defines a Shimura datum $(G^*,X^*)$ (resp. $(G,X)$), where $X^*$ (resp. $X$) denotes the $G^*(\R)$-conjugacy class (resp. $G(\R)$-conjugacy class) of $h: \S \to G^*(\R) \to G(\R)$. The Shimura datum $(G^*,X^*)$ is of PEL-type, whereas the Shimura datum $(G,X)$ is strictly of abelian type.

The action of $h(i)$ defines a complex structure von $V_\R$ and thus a Hodge decomposition $V_\C = V^{-1,0} \oplus V^{0,-1}$ of weight $-1$, with $h(z)$ acting by $z$ (resp. by $\overline{z}$) on $V^{-1,0}$ (resp. on $V^{0,-1}$). The Hodge cocharacter $\nu_h: \G_{m,\C} \to G_\C$ attached to $h$ acts by $z$ (resp. by the identity) on $V^{-1,0}$ (resp. on $V^{0,-1}$).

Both Shimura data $(G^*,X^*)$ and $(G,X)$ have the common reflex field $E_0/\Q$ that is generated over $\Q$ by the coefficients of the characteristic polynomial
\[\det(X_1\alpha_1 + \cdots + X_t\alpha_t | V^{-1,0})\]
for any $\Q$-basis $\{\alpha_j\}$ of $B$.

\subsubsection{Local side}
Fix a prime number $p$ and an embedding $\overline{\Q} \hookrightarrow \overline{\Q}_p$. 

\begin{assumption}
   We assume that the PEL datum that we have fixed so far is unramified at $p$, i.e. that every term in the decomposition
   \[B_{\Q_p} = \prod_{\P|p} B \otimes_F F_\P\]
   is a matrix algebra over an unramified extension of $\Q_p$ (where $\P$ runs over all prime ideals of $F$ lying above $p$, and we denote by $F_\P$ the completion of $F$ with respect to $\P$). 
\end{assumption}

In particular, this assumption implies that $p$ is unramified in $F_c$, and that the reductive group $G$ is unramified at $p$ (which in turn implies that hyperspecial subgroups exist).

We assume to be given the following local data: 
\begin{itemize}
    \item A $*$-stable $\O_{F,(p)}$-order $\O_B \subset B$ (where $\O_{F,(p)} := \O_F \otimes \Z_{(p)}$) such that $\O_B \otimes \Z_p$ is a maximal order in $B_{\Q_p}$, and
    \item an $\O_p$-lattice $\Lambda \subset V_{\Q_p}$ that is self-dual (up to a scalar in $F_p^\times)$ with respect to $\langle\cdot,\cdot\rangle_F$ (where we used the notation $\O_p := \O_F \otimes_\Z \Z_p$ and $F_p := F \otimes_\Q \Q_p$).
\end{itemize}

\begin{example}[Hilbert modular varieties]
    If $B = F$ for a totally real number field $F$ with trivial involution $* = \id$, and let $V= F^2$ with the standard alternating form $\langle v,w\rangle_F = v_1w_2 - v_2w_1$, then $G = \Res_{F/\Q}(\GL_{2,F})$ and the resulting Shimura varieties are the (abelian-type) Hilbert modular varieties (which are of type (C)). Classically, their $\C$-valued points arise from $g$-many copies of the Siegel upper-lower halfplane by quotienting out the twisted diagonal Möbius transformation action of congruence subgroups along the $g$-many different embeddings $F \hookrightarrow \R$.
\end{example}

\begin{example}[Totally indefinite quaternion algebras, cf. \cite{Nek18}, A6.2]
    Let $F$ be a totally real number field of degree $g$, and let $D/F$ be a totally indefinite quaternion algebra over $F$, i.e. $D \otimes_F \R \iso M_2(\R)^g$. Consider $D$ as a semisimple $\Q$-algebra. We can equip it with a positive involution $*$ by defining $d^* = u\overline{d}u^{-1}$, where $\overline{(-)}$ denotes the main involution on $D$, and $u \in D^\times$ such that $\overline{u}= -u$ and $\Nrd(u) = -u^2 \in F_+^\times$. Let $V= D$, equipped with a left action by $D$ via $d\cdot x := xd^*$. $V$ also comes equipped with the skew-symmetric $F$-bilinear form
    \[\langle x,y\rangle_F := \Nrd(xu\overline{y}).\]

    In the end, we obtain quaternionic Shimura varieties of type (C) which fit into our PEL-framework, with $G = \Res_{F/\Q}(D^\times)$. If the quaternion algebra $D$ is not totally indefinite, then the associated quaternionic Shimura variety does not fit into our PEL framework. However, one can still relate them to an auxiliary PEL datum of type (Aeven), see \cite[A6.5-6.9]{Nek18}.
\end{example}

\subsection{Hybrid space}\label{section_moduli}
We keep the notations and assumptions from before. In this subsection, we will introduce the aforementioned hybrid space. 

We use the $\O_p$-lattice $\Lambda \subset V_{\Q_p}$ to define the  hyperspecial level subgroup at $p$
\[K_{hs} := \{g \in G(\Q_p): g\Lambda = \Lambda\}.\]
Furthermore, let $K^p \subset G(\A_f^p)$ be a tame level subgroup, and let $K= K_{hs}K^p$.

From this, we obtain a family of hybrid spaces depending on a parameter $\alpha \in \A_{F,f}^{p,\times}$ as follows.

\begin{definition}\label{hybrid_spaces}
    Let us consider the following \textbf{hybrid moduli problem}. Fix a parameter $\alpha \in \A_{F,f}^{p,\times}$. Let $\M_{\alpha,K}$ be the presheaf of groupoids on the big étale site of schemes over $\O_{E_0} \otimes \Z_{(p)}$, whose value on a scheme $S$ over $\O_{E_0} \otimes \Z_{(p)}$ is the groupoid of abelian schemes with $G$-structure and hybrid $K^p$-level structure, i.e.
    \begin{enumerate}[label=(\roman*)]
        \item $A$ is an abelian scheme over $S$ of relative dimension $g$,
        \item $\lambda: A \to A^\vee$ is a $\Z_{(p)}$-polarization of degree prime to $p$,
        \item $\iota: \O_B \to \End_S(A) \otimes_\Z \Z_{(p)}$ is a morphism respecting the involution $(-)^*$ on $\O_B$ and the Rosati involution on $\End_S(A) \otimes \Z_{(p)}$ coming from the polarization $\lambda$, and
        \item $\overline{\eta}$ is a hybrid $K^{p}$-level structure, meaning the following. Consider the pro-étale sheaf 
        \[\underline{\Isom}_{G}(\underline{H_1}(A,\A_f^p),\underline{V_{\A_f^p}}),\]
        whose sections are $\underline{B \otimes_{\Q} \A_f^p}$-linear isomorphisms $\eta$ that are $\underline{\widehat{\O}_F^{(p)}}$-symplectic similitudes with respect to the lifted Weil pairing $\langle\cdot,\cdot\rangle_\lambda$ on $\underline{H_1}(A,\A_f^p)$ and the pairing $\alpha\langle\cdot,\cdot\rangle_F$ on $\underline{V_{\A_f^p}}$. This is a pro-étale $\underline{G(\widehat{\Z}^{(p)})}$-torsor. Choosing a trivializing cover $\tilde{S} \to S$, $\overline{\eta}$ is a $\underline{K^p}(\tilde{S})$-orbit of a section $\eta$ over $\tilde{S}$, such that $\overline{\eta}$ is invariant under the covering action of $\tilde{S} \to S$.
    \end{enumerate}

    An isomorphism between objects $(A,\iota,\lambda,\overline{\eta})$ and $(A',\iota',\lambda',\overline{\eta}')$ over $S$ is a prime-to-$p$ quasi-isogeny $f: A \to A'$ that is compatible with the $\O_B$-action, the level structure and such that 
    \[f^\vee \circ \lambda' \circ f = c\lambda\]
    for some $c \in \underline{\Z_{(p)}^\times}(S)$. 
\end{definition}

By Lemma \ref{forms_trace}, the condition that $\eta$ is a symplectic similitude can be equivalently phrased in the following ways:
\begin{enumerate}
    \item There exists $c \in \underline{\widehat{\O}_F^{(p),\times}}(S)$ such that 
    \[(x,y)_\lambda = \Tr(c \alpha \langle\eta(x),\eta(y)\rangle_F).\]
    \item There exists $c \in \underline{\widehat{\O}_F^{(p),\times}}(S)$ such that
    \[\alpha(\eta(x),\eta(y))_F = (cx,y)_\lambda\]
\end{enumerate}

When comparing the above hybrid moduli problem to the usual PEL-type moduli problem (e.g. as in \cite[5.2.2]{Zha23}, one notices that the key change lies in the notion of level structure. Indeed, hybrid level structures allow strictly more general symplectic similitudes.

$\M_{\alpha,K}$ is a Deligne-Mumford stack that, for $K^p$ sufficiently small (which we will assume from now on), is representable by a smooth quasi-projective scheme over $\O_{E_0} \otimes\Z_{(p)}$, which we will also denote by $\M_{\alpha,K}$. Note that for any $c \in \widehat{\O}_F^{(p),\times}$, the stacks $\M_{\alpha,K}$ and $\M_{c\alpha,K}$ are canonically isomorphic. 

At the expense of only being defined over $E_0$, we can add level at $p$ to the hybrid moduli problem by fixing a second parameter $\beta \in F_p^\times$ and considering the sheaf
\[\underline{\Isom}_G(\underline{H_1}(A,\Q_p),\underline{V_{\Q_p}})\]
of $\underline{B_{\Q_p}}$-linear isomorphisms that are $\underline{\O_p}$-symplectic similitudes with respect to $\langle\cdot,\cdot\rangle_\lambda$ and $\beta\langle\cdot,\cdot\rangle_F$. If $\alpha \in \A_{F,f}^{p,\times}$ and $\beta \in F_p^\times$, then we will denote by $\M_{\alpha\beta,K}$ the combined moduli space with level $K = K_pK^p$ away from $p$ and at $p$.

For finite level at $p$, the extra parameter $\beta$ is an important piece of information. As soon as we pass to infinite level at $p$ (which we will do very soon), however, the parameter $\beta$ stops being important. Namely, the hybrid spaces with infinite level at $p$ for different parameters $\beta$ can be identified via Hecke action by $G(\Q_p)$. Thus in the following we will mostly stick to the case where $\beta = 1$.

\begin{definition}
    For any sufficiently small level subgroup $K \subset G(\A_f)$, $\M_{\alpha,K}$ is called the \textbf{hybric space} at parameter $\alpha$ and level $K$.
\end{definition}

We have the following notion of Hecke action on our hybrid spaces $\M_{\alpha\beta,K}$ (for $\alpha \in \A_{F,f}^{p,\times}$ and $\beta \in F_p^\times$).

\begin{definition}
    Let $g \in G(\A_f)$ and let $K, K'$ be any levels such that $g^{-1}K' g \subseteq K$, then Hecke action by $g$ is the finite étale map 
    \[[g]: \M_{\alpha\beta,K'} \to \M_{\det(g)\alpha\beta,K}\]
    that sends $(A,\iota,\lambda,\overline{\eta}) \in \M_{\alpha\beta,K}(S)$ to $(A,\iota,\lambda,\overline{g \circ \eta})$.
\end{definition}
If $g \in G(\A_f^p)$, then $\nu(g)\alpha$ is still coprime to $p$. Also, if $g \in \GL_2(\widehat{\O}_F)$, then $[g]$ canonically maps into $\M_{\alpha\beta,K}$ again, however altering the similitude factor in the definition of level structure by $\nu(g)$. If $g^{-1}K'g = K$, then $[g]$ is clearly an isomorphism.

Since the transition maps in the tower $\{\M_{\alpha,K}\}_{K \subset G(\A_f)}$ are finite étale, we can pass to the inverse limit in the category of schemes and let
\[\M_{\alpha} := \invlim_{K} \M_{\alpha,K},\]
which is a scheme locally of finite type, locally Noetherian and regular. 

\subsection{Abelian-type setting}
Now that we have introduced the hybrid space $\M_{\alpha,K}$ for any sufficiently small level subgroup $K \subset G(\A_f)$, we can make the following definition.
\begin{definition}
    The \textbf{polarization action} of $\underline{\O_{F,+}^\times}$ on $\M_{\alpha,K}$ is defined as
    \[\epsilon \cdot (A,\iota,\lambda,\overline{\eta}) := (A,\iota,\epsilon\lambda,\overline{\eta})\]
    for $\epsilon \in \underline{\O_{F,+}^\times}(S)$ and $(A,\iota,\lambda,\overline{\eta}) \in \M_K(S)$.
\end{definition}

We alert the reader that in order to define the polarization action, we really needed to pass to the hybrid space. Indeed, polarization action by $\epsilon \in \O_{F,+}^\times$ will change the similitude factor in the definition of hybrid level structure by a factor $\epsilon$ (due to the change of Weil pairing). If we were considering the usual PEL-type level structure, then this could change the rational similitude factor into a non-rational one.

If we are in type (C), this polarization action factors through the finite quotient
\[\Delta(K) := \O_{F,+}^\times/(\O_{F}^\times \cap K)^2\]
since for $\epsilon = \epsilon'^2$ with $\epsilon' \in \O_F^\times \cap K$, multiplication by $\iota(\epsilon')$ defines an isomorphism
\[[\iota(\epsilon')]: (A,\iota,\lambda,\overline{\eta}) \overset{\sim}{\to} (A,\iota,\epsilon^{-1}\lambda,\overline{\eta}).\]
If we are in type (A even), similarly the polarization factors through the finite quotient
\[\Delta(K) := \O_{F,+}^\times/N_{F_c/F}(\O_{F_c}^\times \cap K).\]

The action of $\Delta(K)$ permutes the connected components of $\M_{\alpha,K}$ and the stabilizer of each connected component is
\[(\nu(K) \cap \O_{F,+}^\times)/(K \cap \O_F^\times)^2\]
in type (C) resp.
\[(\nu(K) \cap \O_{F,+}^\times)/N_{F_c/F}(K \cap \O_{F_c}^\times)\]
in type (A even) (see \cite{Nek18}).

Thus the polarization action by $\Delta(K)$ is free if and only if the above quotients are trivial. It is a subtle question under which circumstances this happens. To give an answer, we first recall the following lemma which is originally due to Chevalley.

\begin{lemma}[\cite{Tay03}, 2.1]
    Let $L$ be a number field and $S$ a finite set of places of $L$. By $L_S^\times$ we will denote the subgroup of $L^\times$ consisting of elements that are units at all finite places $v \notin S$ and positive at all real places $v \notin S$. Then for any positive integer $n \in \Z_{>0}$, there is an open subgroup $U \subset \prod_{v \notin S \text{ finite}} \O_{L,v}^\times$ such that $L_S^\times \cap U \subset (L_S^\times)^n$.
\end{lemma}

\begin{corollary}
    Let $L$ be a number field. Every finite index subgroup of $\O_L^\times$ contains a subgroup of the form $U \cap \O_L^\times$, where $U \subseteq \widehat{\O}_L^\times$ is a compact open subgroup with $U_\P = \O_{\P}^\times$ for all $\P$ lying above $p$.
\end{corollary}
\begin{proof}
    Let $A \subset \O_L^\times$ be a finite index subgroup. Let $S$ be the union of all infinite places and all finite places dividing $p$, and let $n \gg 0$ be such that for every $x \in \O_L^\times$, $x^n \in A$.
    
    Choose $U$ as in Chevalley's lemma and fill it up with $U_\P = \O_{\P}^\times$ for any prime $\P$ above $p$. Then
    \[\O_L^\times \cap U = \O_L^\times \cap L_S^\times \cap U \subset \O_L^\times \cap (L_S^\times)^n \subset (\O_L^\times)^n \subset A.\qedhere\]
\end{proof}

With this corollary at hand, we can give a first negative example which shows that the polarization action will in general not turn out to be free, even if we are willing to go into deep enough level away from $p$.

\begin{example}
    Consider the case of Hilbert modular varieties, i.e. we are in a type (C) situation with $G= \Res_{F/\Q}(\GL_{2,F})$. Assume we are given hyperspecial level at $p$, and the principal level-$N$-subgroup $K(N)$ for some $N$ prime to $p$. Then the polarization action is free if and only if the quotient $(\det K(N) \cap \O_{F,+}^\times)/(K(N) \cap \O_F^\times)^2$ is trivial. Note that $K(N) \cap \O_F^\times = (1+N\O_F)^\times$ and $\det K(N) \cap \O_{F,+}^\times = (1+N\O_F)^\times_+$. So we need to analyze the quotient
    \[(1+N\O_F)_+^\times /(1+N\O_F)^{\times 2}.\]
    We claim that for $N$ sufficiently large, we have $(1+N\O_F)_+^\times = (1+N\O_F)^\times$, which would imply that the quotient is isomorphic to $(\Z/2\Z)^{g-1}$ by Dirichlet's unit theorem. 

    To prove the claim, we apply the above corollary to the finite index subgroup $\O_{F,+}^\times \subset \O_F^\times$. Thus there is a compact open subgroup $U\subseteq \widehat{\O}_F^\times$ with $U_\P = \O_\P^\times$ for all $\P$ lying above $p$, such that $U \cap \O_F^\times \subseteq \O_{F,+}^\times$. In particular, for any $N \gg 0$ that is prime-to-$p$ such that $K(N) \subset U$, we see that
    \[(1+N\O_F)^\times \subset \O_{F,+}^\times,\]
    i.e. $(1+N\O_F)_+^\times = (1+N\O_F)^\times$, proving the claim.
\end{example}

However, not all hope is lost. The same corollary of Chevalley's Lemma can be used to find certain level subgroups for which the polarization action turns out to be free.

\begin{lemma}\label{magic_lemma}
    Let $K= K_pK^p \subset G(\A_f)$ be any compact open subgroup. If we are in type (C), then there exists an open compact normal subgroup $K'^p \subseteq K^p$ of finite index such that for any wild level $K'_p$ containing $K_p$,
    \[\nu(K'_pK'^p) \cap \O_{F,+}^\times = (K'_pK'^p \cap \O_F^\times)^2.\]
    Similarly, if we are in type (A even), the same is true for the formula
    \[\nu(K'_pK'^p) \cap \O_{F,+}^\times = N_{F_c/F}(K'_pK'^p \cap \O_{F_c}^\times).\]
\end{lemma}
\begin{proof}
    We will only consider the case of type (C), the case of type (A even) can be handled analogously.

    First note that for any compact open subgroup $K \subset G(\A_f)$, the inclusion $(K \cap \O_F^\times)^2 \subset \nu(K) \cap \O_{F,+}^\times$ always holds, so that we only have to deal with the reverse inclusion. 
    
    We apply the corollary of Chevalley's Lemma to the finite index subgroup $(K_pK^p \cap \O_F^\times)^2 \subset \O_F^\times$. Thus we obtain a compact open subgroup $U \subseteq \widehat{\O}_F^\times$ with $U_\P = \O_\P^\times$ for all $\P$ lying above $p$, such that $U \cap \O_F^\times \subseteq (K_pK^p \cap \O_F^\times)^2$. This implies that for any $K'_p$ containing $K_p$
    \[U \cap \nu(K'_pK^p) \cap \O_{F,+}^\times \subset U \cap \O_F^\times \subset U \cap (K'_pK^p \cap \O_F^\times)^2.\]
    Now let $K'^p = \nu^{-1}(U^p) \subset K^p$. Then $U \cap (K'_pK^p \cap \O_F^\times)^2 \subset (K'_pK'^p \cap \O_F^\times)^2$, and $\nu(K'_pK'^p) \subset U \cap \nu(K'_pK^p)$.
\end{proof}

While the negative example showed that we cannot expect the polarization action to \textit{always} be free eventually if we are only willing to pass to deep enough level, Lemma \ref{magic_lemma} makes it clear that we can at least find a cofinal system of compact open subgroups for which the polarization action is in fact free. We will call such level subgroups \textbf{good} level subgroups from now on.

Lemma \ref{magic_lemma} is even a bit stronger than that. While it is most likely not true that for any tame level $K^p$ we can find a finite index subgroup $K'^p \subset K^p$ such that $K_pK'^p$ is good \textit{uniformly for any} $K^p$, we at least know that once we have given ourselves a wild level subgroup $K_p \subset G(\Q_p)$ to start with, we can find a uniform level subgroup $K'^p \subset K^p$ such that $K'_pK'^p$ is good for any wild level subgroup $K'_p \supset K_p$ (in particular hyperspecial level subgroups containing $K_p$). 

\begin{definition}
    A level subgroup $K_pK^p \subset G(\A_f)$ is called \textbf{uniformly good} if for any wild level subgroup $K'_p \supset K_p$, the level subgroup $K_p'K^p$ is good.  
\end{definition}

Note that the property of being uniformly good is stable under conjugation. This follows from the fact that the condition $\nu(K) \cap \O_{F,+}^\times = (K \cap \O_F^\times)^2$ is clearly invariant under conjugation (similarly for type (A even)). Thus Hecke action will preserve the set of uniformly good level structures.

All in all, we find that the system of uniformly good level subgroups is a cofinal system in the system of all level subgroups. Choosing such a uniformly good level subgroup $K$, the quotient scheme $\M_{\alpha,K}/\Delta(K)$ exists and is quasi-projective and smooth, either over $E_0$ or $\O_{E_0} \otimes \Z_{(p)}$ (depending on if we chose hyperspecial level at $p$ or not).

Thus we can make the following definition.
\begin{definition}
    For any sufficiently small uniformly good level subgroup $K \subset G(\A_f)$, we set $\M_{G,\alpha,K} := \M_{\alpha,K}/\Delta(K)$.
\end{definition}

From these quotient spaces, we can obtain a model for the full abelian-type Shimura variety we are interested in. Namely, fix a set $\Psi$ of representatives of the (finite) double coset space
\[\O_{F,(p),+}^\times \backslash \A_{F,f}^{p,\times} /\widehat{\O}_{F}^{(p),\times}.\]
Fix a uniformly good level structure $K$ with hyperspecial level at $p$. Then the disjoint union
\[\coprod_{\alpha \in \Psi} \M_{G,\alpha,K} = \coprod_{\alpha \in \Psi} \M_{\alpha,K}/\Delta(K)\]
constitutes a smooth quasi-projective integral model of the abelian-type Shimura variety at level $K$ over $\O_{E_0} \otimes \Z_{(p)}$. We can proceed similarly if we add level at $p$ to the situation.

Thus it suffices to construct a hybrid Igusa stack for each parameter $\alpha \in \A_{F,f}^{p,\times}$ separately, and then take their disjoint union accordingly. 
\begin{convention}
    From now on, if the parameter $\alpha \in \A_{F,f}^{p,\times}$ is understood from context or does not play a role in the discussion, we will drop it from our notation.
\end{convention}

\section{Perfectoid spaces, diamonds and v-stacks}
In this subsection, we will briefly review the theory of perfectoid spaces, diamonds and the v-topology in order to fix our notation. 

\subsection{Perfectoid spaces}
\begin{definition}
    A topological ring $R$ is called a \textbf{Huber ring} if it admits an open subring $R_0$ that is adic with respect to a finitely generated ideal $I$. A Huber ring is called \textbf{Tate} if it contains a \textbf{pseudo-uniformizer}  $\varpi \in R$, i.e. a topologically nilpotent unit.
\end{definition}

\begin{definition}
    A Tate ring $R$ is called \textbf{perfectoid} if it is complete, uniform (i.e. the subring of power-bounded elements $R^\circ \subset R$ is bounded) and there exists a pseudo-uniformizer $\varpi \in R$ such that $\varpi^p$ divides $p$ in $R^\circ$ and the Frobenius map
    \[\Phi: R^\circ/\varpi \to R^\circ/\varpi^p\]
    is an isomorphism.
\end{definition}

\begin{definition}
    Let $R$ be a perfectoid Tate ring. The \textbf{tilt} of $R$ is the topological ring
    \[R^\flat = \invlim_{x \mapsto x^p} R\]
\end{definition}

Let us expand on this a little bit. The topology of $R^\flat$ is the inverse limit topology. The ring structure is given as follows: First, we have multiplicative bijective maps

\[\invlim_{\Phi} R^\circ \longrightarrow \invlim_\Phi R^\circ/p \longrightarrow \invlim_\Phi R^\circ/\varpi\]
that equip the left-hand side with a ring structure. The inverse maps are given by mapping $x = (x_0,x_1,...)$ to $(x^\sharp, (x^{1/p})^\sharp, (x^{1/{p^2}})^\sharp,...) \in \invlim_{\Phi} R^\circ$, where $x^\sharp = \lim_{i \to \infty} \tilde{x_i}^{q^i}$, the $\tilde{x_i}$ being lifts of $x_i$ to $R^\circ$.

Now take any preimage $\varpi^\flat$ of $\varpi$ under the map
\[\invlim_{\Phi} R^\circ \iso \invlim_\Phi R^\circ/\varpi^p \to R^\circ/\varpi^p\]
This way we see that $R^\flat$ is complete, Tate and even perfectoid $\F_p$-algebra with ring of power-bounded elements $R^{\flat\circ} \iso \invlim_{\Phi} R^\circ$ and pseudo-uniformizer $\varpi^\flat$, so in particular $R^\flat = R^{\flat\circ}[1/{\varpi^\flat}]$.

The projection to the first component map $R^\flat \to R, x \mapsto x^\sharp$ (this coincides with the above notion of $(-)^\sharp$) induces a ring isomorphism $R^{\flat\circ}/\varpi^\flat \iso R^\circ/\varpi$. Additionally we have the following

\begin{lemma}
The set of open, bounded and integrally closed subrings $R^+ \subset R^\circ$ is in bijection with the set of integral subrings $R^{\flat +} \subset R^{\flat\circ}$, via $R^+ \mapsto R^{\flat +} := \invlim_{x \mapsto x^p} R^+$. Furthermore $R^{\flat +}/\varpi^\flat \iso R^+/\varpi$.
\end{lemma}

\begin{definition}\label{thp}
    An \textbf{affinoid (perfectoid) Tate-Huber pair} is a pair $(R,R^+)$, where $R$ is a (perfectoid) Tate ring and $R^+ \subset R^\circ$ is an open, bounded, integrally closed subring. A morphism of affinoid Tate-Huber pairs $(R,R^+) \to (R', R'^+)$ is a morphism of topological rings that sends $R^+$ into $R'^+$. The tilt of an affinoid perfectoid Tate-Huber pair $(R,R^+)$ is the affinoid perfectoid Tate-Huber pair $(R^\flat,R^{\flat +})$.
\end{definition}

\begin{definition}
    A \textbf{perfectoid space} $X$ is an adic space that admits a cover by affinoid opens that are isomorphic to affinoid adic spaces of the form $\Spa(R,R^+)$, where $(R,R^+)$ is an affinoid perfectoid Tate-Huber pair.
\end{definition}

To a perfectoid space $X$ we can functorially associate its tilt $X^\flat$ by tilting its rational open subsets and then gluing. There is a homeomorphism $|X| \iso |X^\flat|$ that preserves rational open subsets.

\begin{example}
    Given a non-Archimedean complete algebraically closed field $C$ with ring of integers $\O_C$ and integral subring $C^+ \subset \O_C$, then $\Spa(C,C^+)$ is perfectoid. We will call this a \textbf{geometric point} in analogy with the scheme-theoretic setting. If $C^+ = \O_C$, we call it a \textbf{rank 1 geometric point}.
\end{example}

\begin{definition}
    Let $X$ be a perfectoid space in characteristic $p$. An \textbf{untilt of $X$} is a pair $(X^\sharp, \iota)$, consisting of a perfectoid space $X^\sharp$ (of any characteristic), and a fixed isomorphism $\iota: (X^\sharp)^\flat \iso X$.
\end{definition}
To simplify notation we will often drop the explicit isomorphism $\iota$ and just write $X^\sharp$ for an untilt of $X$.

\begin{remark}\label{multiple_untilts}
    Assume $\Spa(R,R^+)$ is affinoid perfectoid in characteristic $p$ with pseudo-uniformizer $\varpi \in R^+$, and we are given two untilts $\Spa(R^{\sharp_1},R^{\sharp_1+})$ and $\Spa(R^{\sharp_2}, R^{\sharp_2+})$. Then for $i \in \{1,2\}$, there is Fontaine's theta map 
    $\theta: W(R^+) \to R^{\sharp_i+}$
    whose kernel is generated by an element $\xi$. We may replace $\varpi$ by one of its $p^n$-th roots to achieve that $\xi = p + [\varpi]a$ for some $a \in W(R^+)$. By quotienting out $(p,[\varpi])$, this yields a map
    \[R^{\sharp_i+} = W(R^+)/(p+[\varpi]a) \twoheadrightarrow R^+/\varpi\]
    
    Choosing $\varpi$ sufficiently small, we can hence simultaneously get two quotient maps
    \[R^{\sharp_i+} \twoheadrightarrow R^+/\varpi\]
    for $i =1,2$. This observation will be very important later when defining the hybrid Igusa stack.
\end{remark}

\subsection{The pro-étale and v-topology}
Recall the following classes of morphisms between perfectoid spaces.

\begin{definition}
    Let $f: X \to Y$ be a morphism of perfectoid spaces.
    \begin{enumerate}[label=(\roman*)]
        \item $f$ is \textbf{finite étale} if for any affinoid perfectoid open $V = \Spa(R,R^+) \subset Y$, the preimage $f^{-1}(V) = \Spa(S,S^+)$ is again affinoid perfectoid and the induced ring map $R \to S$ is finite étale.
        \item $f$ is \textbf{étale} if for any $x \in X$ there is an open neighbourhood $V \subset X$ of $x$, an affinoid perfectoid open $U = \Spa(R,R^+) \subset Y$ such that $f(V) \subset U$ and such that $f|_V$ factors as a composition of an open immersion followed by a finite étale morphism.
        \item $f$ is \textbf{pro-étale} if for any $x \in X$ there is an affinoid perfectoid open $V \subset X$ of $x$, an affinoid perfectoid open $U = \Spa(R,R^+) \subset Y$ such that $f(V) \subset U$ and $f$ can be written as a limit, over a small cofiltered index category $I$, of étale maps $V_i \to U$ with $V_i$ being affinoid perfectoid for all $i \in I$.
    \end{enumerate}
\end{definition}

Denote by $\Perfd$ the category of perfectoid spaces, and by $\Perf$ the full subcategory of perfectoid spaces in characteristic $p$. 

\begin{definition}
    Recall the following Grothendieck topologies on $\Perf$ (or $\Perfd$):
    \begin{enumerate}[label=(\roman*)]
        \item The \textbf{v-topology} on $\Perf$ is the Grothendieck topology which has as coverings jointly surjective maps $\{f_i: X_i \to Y\}_{i \in I}$ such that for every quasicompact open $U \subset Y$ there is a finite subset $J \subset I$ and quasicompact open $V_j \subset X_j$ for all $j \in J$ such that $\bigcup_{j \in J} f(V_j) = U$.
        \item The \textbf{pro-étale topology} on $\Perf$ is the Grothendieck topology which has as coverings those v-covers $\{f_i: X_i \to Y\}_{i \in I}$ such that all $f_i$ are pro-étale.
    \end{enumerate}
\end{definition}
We call $\Perf$ (or $\Perfd)$ equipped with the pro-étale resp. v-topology the \textbf{big pro-étale} resp. \textbf{v-site}. There are some set-theoretic issues arising here, which can be avoided by using cut-off cardinals. Here, we will avoid these issues by simply ignoring them.

By \cite[8.6, 8.7]{Sch22} both the v-topology and the pro-étale topology on $\Perf$ and $\Perfd$ are subcanonical. Thus we will always confuse a perfectoid space distinguish between a perfectoid space and the v-sheaf it represents.

We can also consider the small pro-étale site $X_{\text{pro-ét}}$ of a perfectoid space $X$. It has as objects perfectoid spaces $Y$ that are pro-étale over $X$, and as coverings it has jointly surjective morphisms with the same quasicompactness condition as in the definition of the big pro-étale site. 

Within the v-topology, it is possible to basically rip apart perfectoid spaces into points. More precisely, the subclass of strictly totally disconnected perfectoid spaces form a basis for the v-topology. Recall that we call a perfectoid space $X$ \textbf{totally} (resp. \textbf{strictly totally}) \textbf{disconnected} if it is quasicompact, quasiseparated and every open (resp. étale) cover of $X$ splits. 

\begin{prop}[\cite{Sch22}, 1.15]
    A perfectoid space $X$ is (strictly) totally disconnected if and only if it is affinoid and every connected component of $X$ is of the form $\Spa(K,K^+)$ with $K$ an (algebraically closed) perfectoid field and $K^+$ an open bounded valuation subring.
\end{prop}

\subsection{v-stacks}
The Igusa stack is supposed to be a stack on the big v-site $\Perf$, so for convenience we briefly recall the relevant notions.
\begin{definition}
    A \textbf{v-stack} is a contravariant 2-functor $F: \Perf \to \Grpd$ into the 2-category $\Grpd$ of groupoids, that satisfies descent for v-covers. 

    Concrete, satisfying descent means the following. Let $X \to Y$ be any v-cover of perfectoid spaces. Write $p_1, p_2: X \times_Y X \to X$ and $p_{12},p_{23},p_{13}: X \times_Y X \times_Y X \to X \times_Y X$ for the projections. Denote by $F(X/Y)$ the category of descent data. Its objects are tuples $(s,\alpha)$ with $s \in F(X)$ and $\alpha: p_1^*s \overset{\sim}{\rightarrow} p_2^*s$ an isomorphism that satisfies the cocycle condition
    \[p_{23}^*\alpha \text{ } \circ p_{12}^*\alpha = p_{13}^*\alpha.\]

    Then a contravariant 2-functor $F: \Perf \to \Grpd$ satisfies descent for v-covers if, for any v-cover $X \to Y$ of perfectoid spaces, the natural functor $F(Y) \to F(X/Y)$ is an equivalence of categories. 
\end{definition}

All v-stacks $X$ that we are going to encounter in this thesis will be \textbf{small}, which means that they admit a presentation
\[R = Y \times_X Y \rightrightarrows Y \to X,\]
where $Y$ is a perfectoid space, and $Y \times_X Y$ is a small v-sheaf, i.e. admits a surjection of v-sheaves from a perfectoid space. 

The presentation $R \rightrightarrows Y \to X$ allows us to define the underlying topological space of a small v-stack. Namely, let $\tilde{R} \to R$ be a surjection from a perfectoid space, then $|X|:= |Y|/|\tilde{R}|$. This is independent of any choices made (see \cite[12.7,12.8]{Sch22}).

The Igusa stack for our abelian-type Shimura variety will arise from the following general procedure.
\begin{example}
    For any topological space $T$ we can define the v-sheaf $\underline{T}: S \mapsto \Hom(|S|,T)$ on $\Perf$. Assume we are given a v-sheaf $X$ on $\Perf$ equipped with an action by a topological group $G$, we can consider the stacky quotient $[X/\underline{G}]$, which is the v-sheaf theoretic coequalizer of the projection and action map
    \[X \times \underline{G} \rightrightarrows X\]
    This is a v-stack. If $G$ is locally profinite and $X$ is a perfectoid space, then $X \times \underline{G}$ is representable by a perfectoid space. Also both the projection and action map are pro-étale. If $X$ is a small v-sheaf, then $[X/\underline{G}]$ is a small v-stack.
\end{example}

\begin{definition}
    Given a diagram $X \overset{f}{\to} Z \overset{g}{\leftarrow} Y$ of small v-stacks, its \textbf{fiber product} $X \times_Z Y$ is the presheaf of groupoids that sends $S \in \Perf$ to the groupoid of triples $(x,y, \varphi: f(x) \overset{\sim}{\to} g(y))$ with $x \in X(S)$, $y \in Y(S)$ (a morphism $(x,y,\varphi) \to (x',y',\varphi')$ of such triples being pairs of maps $(\alpha: x \to x', \beta: y \to y')$ such that $\varphi' \circ f(\alpha) = g(\beta) \circ \varphi$).
\end{definition}
By \cite[12.10]{Sch22}, $X \times_Z Y$ is again a small v-stack. 

\begin{definition}
    A v-stack $X$ is \textbf{quasicompact} if there is a surjection of v-stacks $Y \twoheadrightarrow X$ with $Y$ an affinoid perfectoid space. 
\end{definition}
In particular if $X$ is quasicompact, then $X$ is a small v-stack and its underlying topological space $|X|$ is quasicompact.

\begin{definition}[\cite{Sch22}, 10.7]
    Let $f: Y \to X$ be a morphism of v-stacks. We call the morphism
    \begin{enumerate}[label=(\roman*)]
        \item \textbf{0-truncated} if for all $S \in \Perf$ the morphism of groupoids $f(S): Y(S) \to X(S)$ is faithful. Equivalently, if the diagonal map $\Delta_f: Y \to Y \times_X Y$ is an injection. 
        \item \textbf{quasicompact} if for any affinoid perfectoid space $S$ with a morphism $S \to X$, the fiber product $Y \times_X S$ is quasicompact.
        \item \textbf{quasiseparated} if the diagonal $\Delta_f$, which is 0-truncated, is qcqs.
        \item an \textbf{open} (resp. \textbf{closed}) immersion if for every (totally disconnected) perfectoid space $T$ with a morphism $T \to X$, the pullback $Y \times_X T \to T$ is represented by an open (resp. closed) immersion.
        \item \textbf{separated} if the diagonal $\Delta_f$ is a closed immersion (thus automatically 0-truncated).
        \item \textbf{partially proper} if it is separated and for every diagram
\[\begin{tikzcd}
	{\Spa(R,R^\circ)} && Y \\
	\\
	{\Spa(R,R^+)} && X
	\arrow[from=1-1, to=1-3]
	\arrow["f", from=1-3, to=3-3]
	\arrow[from=3-1, to=3-3]
	\arrow[from=1-1, to=3-1]
	\arrow[dashed, from=3-1, to=1-3]
\end{tikzcd}\]   
    where $R$ is any perfectoid Tate ring with an open, bounded and integrally closed subring $R^+ \subset R$, there exists a unique dotted arrow making the diagram commute.
    \end{enumerate}
\end{definition}


\subsection{Diamonds}
\begin{definition}
    A \textbf{diamond} is a pro-étale sheaf $\D$ on $\Perf$ that can be written as $X/R$ with $X$ and $R$ being perfectoid spaces and $R \subset X \times X$ being a pro-étale equivalence relation. 
\end{definition}

By \cite[11.9]{Sch22}, diamonds are small v-sheaves, hence in the context of diamonds and morphisms between them we have the same terminology as developed in the last section. 

One of the most important examples of a diamond is the following.
\begin{example}
    Consider the perfectoid field $\Q_p^{\cycl} := \widehat{\Q_p(\mu_{p^\infty})}$. We define
    \[\Spd\Q_p := \coeq(\Spa(\Q_p^{\cycl})^\flat \times \underline{\Gal(\Q_p^{\cycl}/\Q_p)} \rightrightarrows \Spa(\Q_p^{\cycl})^\flat)\]
    This is a diamond with underlying topological space being a point. 
\end{example}

\begin{theorem}[\cite{SW20}, 8.4.2]
    The category of perfectoid spaces over $\Spa\Q_p$ is equivalent to the category of perfectoid spaces $X$ in characteristic $p$, equipped with a structure morphism $X \to \Spd \Q_p$ as sheaves on $\Perf$.
\end{theorem}

In order to turn our hybrid space and our abelian-type Shimura variety into objeccts of $p$-adic geometry, we need to explain how to attach diamonds or v-sheaves to various (formal) schemes.

First, assume $X$ is an analytic adic space over $\Spa\Z_p$ (meaning that it is locally Tate, see \cite[4.3.1]{SW20}). Then we can define a presheaf $X^\diamond$ on $\Perf$ by
\[T \mapsto X^\diamond(T) := \{(T^\sharp, T^\sharp \to X)\}/\iso\]
where $T^\sharp$ is an untilt of $T$ and $T^\sharp \to X$ is a map of adic spaces.
By \cite[10.1.5]{SW20} the presheaf $X^\diamond$ is a diamond. In fact, if $X$ is perfectoid then $X^\diamond = X^\flat$, and more generally if $X$ is analytic, one can cover it with a perfectoid space by a perfectoid equivalence relation in the pro-étale topology, and then tilt both. If $X = \Spa(R,R^+)$ we will write $X^\diamond = \Spd(R,R^+)$, and we will also drop $R^+$ from the notation if $R^+ = R^\circ$.

In the case of a pre-adic space $X$ over $\Spa \Z_p$, one can still consider the above presheaf (with maps of adic spaces replaced by maps of pre-adic spaces), but it does not necessarily yield a diamond, though by \cite[18.1.1]{SW20} it always yields a v-sheaf. In particular this applies to formal schemes $\X$. If $\X = \Spf A$ is formal affine, then $\X^\diamond = \Spd(A,A)$, where $\Spa(A,A)$ is the pre-adic space attached to $\Spf A$. This globalizes and  defines a functor from formal schemes to v-sheaves. 

For schemes over $\Z_p$, there are two possible ways of turning them into diamonds (or v-sheaves), and in general they won't agree. First assume $X= \Spec A$ is affine.
\begin{enumerate}[label=(\roman*)]\label{dia_schemes}
    \item The \textbf{small diamond} $X^\diamond$ attached to $X$ is the v-sheaf on $\Perf$
    \[S \mapsto \{(S^\sharp, f: A \to \O_{S^\sharp}^+(S^\sharp)\}\]
    where $S^\sharp$ is an untilt of $S$ and $f$ is a ring homomorphism. 

    \item The \textbf{big diamond} $X^\Diamond$ attached to $X$ is the v-sheaf on $\Perf$
    \[S \mapsto \{(S^\sharp, f: A \to \O_{S^\sharp}(S^\sharp)\}\]
    where $S^\sharp$ is an untilt of $S$ and $f$ is a ring homomorphism.
\end{enumerate}
As discussed in \cite[2.2]{AGLR22}, these constructions glue to yield functors from the category of schemes over $\Z_p$ to v-sheaves, however in general for non-affine schemes we won't have a nice moduli description as above.

For a scheme $X$ locally of finite type over $\Z_p$, the two diamond functors can be described via $p$-adic completion and adic analytification:
\begin{enumerate}[label=(\roman*)]
    \item $X^\diamond$ is obtained by first $p$-adically completing $X$, which yields a formal scheme over $\Z_p$, and then applying the diamond functor for formal schemes.
    \item $X^\Diamond$ is obtained by applying adic analytification to $X$, which yields an adic space over $\Spa\Z_p$, and then taking its associated diamond.
\end{enumerate}

From the first description of the two diamond functors, it is clear that we have a natural map $X^\diamond \to X^\Diamond$. It is an open immersion if $X$ is separated, and an isomorphism if $X$ is proper.

\section{Passage to the adic world}
In this section, we will introduce multiple incarnations of the hybrid spaces and abelian-type Shimura varieties, introduced in Section \ref{section_sv}, as objects in $p$-adic geometry. We will always assume that $K \subset G(\A_f)$ is a level subgroup, sometimes additionally assumed to be hyperspecial at $p$, and always assumed to be uniformly good when talking about the abelian-type Shimura varieties. Recall that we drop the parameter $\alpha \in \A_{F,f}^{p,\times}$ from our notation unless necessary.

Recall from the introduction that in order to state Conjecture \ref{conjecture}, we needed to fix an isomorphism $\C \iso \overline{\Q}_p$. Let $v$ be the prime of the reflex field $E_0$ above $p$ that is induced by the composition $E_0 \to \C \iso \overline{\Q}_p$. We complete $E_0$ with respect to $v$, which yields the local field $E$, a finite extension of $\Q_p$. Via base change, we obtain the hybrid space $\M_{K}$ and the abelian-type Shimura variety $\M_{G,K}$ over $E$ (respectively their integral models over $\O_E$ if $K$ is hyperspecial at $p$). We will denote these spaces again by $\M_K$ resp. $\M_{G,K}$.

\begin{definition}
    Let $K$ be hyperspecial at $p$. By $\X_K$ we denote the formal scheme over $\Spf\O_E$ arising via $p$-adic completion from $\M_K$, and call it the \textbf{formal hybrid space}. Analogously, we will call $\X_{G,K}$ the \textbf{formal abelian-type Shimura variety}.

    The \textbf{formal hybrid space at infinite level} is the inverse limit in the category of formal schemes $\X = \invlim_{K} \X_K$, where $K$ runs over all hyperspecial level subgroups. We similarly define the \textbf{formal abelian-type Shimura variety at infinite level}.
\end{definition}

The formal hybrid space $\X_{K}$ still has a moduli description in terms of formal abelian schemes with $G$-structure and hybrid $K^p$-level structure. We can also $p$-adically complete the universal abelian variety over $\M_K$ to obtain a universal formal abelian scheme $G$-structure and hybrid $K^p$-level structure.

\begin{definition}
    The \textbf{diamond hybrid space} is the diamond $\cX_K = \M_K^\Diamond$ over $\Spd E$. Similarly, the \textbf{diamond abelian-type Shimura variety} is the diamond $\cX_{G,K} = \M_{G,K}^\Diamond$ over $\Spd E$.
\end{definition}

We can also apply the diamond functor to the universal abelian variety $\cA_{K}/\M_K$ to obtain a proper map of diamonds
\[\pi: \cA_K^\diamond \to \cX_K.\]

\begin{definition}
    The ($p$-adic) \textbf{Tate module} of the universal abelian variety $\cA_K$ is the sheaf of $\underline{\Z_p}$-modules on $\cX_K$
    \[T_p(\cA_K) := \underline{\Hom}_{\underline{\Z_p}}(R^1\pi_* \underline{\Z_p},\underline{\Z_p}).\]
\end{definition}
Via the $\O_B$-action on $\cA_K$, the Tate module also comes equipped with a natural $\O_B$-module structure. 

\begin{definition}
    The \textbf{hybrid space at infinite level at $p$} is the diamond (here $K$ is hyperspecial at $p$)
    \[\cX_{K^p} := \underline{\Isom}_G(T_p(\cA_K),\underline{\Lambda}) \to \cX_{K},\]
    where the isomorphism sheaf is of $\O_B \otimes \Z_p$-linear trivializations of the Tate module that are $\underline{\O_p}$-symplectic similitudes with respect to the lifted Weil pairing $\langle\cdot,\cdot\rangle_\lambda$ on the Tate module and $\langle\cdot,\cdot\rangle_F$ on $\underline{\Lambda}$.
\end{definition}
Alternatively, we can write the hybrid space at infinite level at $p$ as
\[\cX_{K^p} = \invlim_{K_p} \cX_{K_pK^p}\]
in the category of diamonds. 

\begin{definition}
    The \textbf{hybrid space at infinite level} is the diamond
    \[\cX = \invlim_{K^p} \cX_{K^p} = \invlim_K \cX_K.\]
\end{definition}

Our choice of asymmetry in the definition of the hybrid space at infinite level can be explained by our main goal: we want to break apart its $p$-adic geometry via the Hodge-Tate period map and the construction of an Igusa stack. Thus, it makes sense to pay special attention to the infinite level at $p$.

When passing to the abelian-type Shimura varieties, we will completely drop this asymmetry. This is due to the subleties surrounding uniformly good level structures.

\begin{definition}
    The \textbf{abelian-type Shimura variety at infinite level} is the inverse limit
    \[\cX_G := \invlim_{K} \cX_{G,K},\]
    in the category of diamonds, where the inverse limit runs over all uniformly good level subgroups $K \subset G(\A_f)$ (whose $p$-part contains $K_{hs}$).
\end{definition}

Last but not least, we introduce the good reduction loci. First assume that the level subgroup $K \subset G(\A_f)$ is hyperspecial at $p$, then we define
\[\cX_{K}^\circ = (\X_K)_{\eta}^\diamond = \X_K^\Diamond \times_{\Spd \O_E} \Spd E,\]
where by $(-)_\eta$ we denote taking adic generic fiber of a formal scheme. Alternatively, $\cX_K^\circ = \M_K^\diamond$ in terms of the small diamond functor. For any diamond $D$ over $\cX_K$, we define its good reduction locus via pullback as
\[D^\circ = D \times_{\cX_K} \cX_K^\circ,\]
which in particular includes the diamonds $\cX_{K_pK^p}$, $\cX_{K^p}$ and $\cX$. For later use, we remark that $\cX_{K^p}^\circ = \invlim_{K_p} \cX_{K_pK^p}^\circ$ and $\cX^\circ = \invlim_{K_pK^p} \cX^\circ_{K_pK^p}$. For the latter claim, we use that for $K' \subset K$ both hyperspecial at $p$, $\cX_{K}^\circ$ pulls back to $\cX_{K'}^\circ$.

In the abelian-type situation, we can similarly define $\cX_{G,K}^\circ$ for any uniformly good level subgroup $K \subset G(\A_f)$ that is hyperspecial at $p$. For any uniformly good level subgroup $K_pK^p$, we define
\[\cX_{G,K_pK^p}^\circ = \cX_{G,K_pK^p} \times_{\cX_{G,K_{hs}K^p}} \cX_{G,K_{hs}K^p}^\circ,\]
and at infinite level we get
\[\cX_G^\circ = \invlim_{K} \cX_{G,K}^\circ,\]
where $K$ runs over all uniformly good level subgroups.

For later use, we want to see that $\cX^\circ$ still retains a moduli description, in the sense that it is the analytic sheafification of a presheaf which is given in terms of a moduli description.

\begin{lemma}\label{grlocus_moduli}
    Let $K$ be hyperspecial at $p$. The good reduction locus $\cX_{K}^\circ$ is the analytic sheafification of the presheaf $\cX^{\circ,\pre}$ on $\Perf$
    \[S = \Spa(R,R^+) \mapsto \{(S^\sharp, \Spf R^{\sharp +} \to \X_{K})\},\]
    where $S^\sharp = \Spa(R^\sharp, R^{\sharp +})$ is an untilt of $S$ over $\Spa E$.

    Similarly, $\cX_{K_{hs}}^\circ = \invlim_{K^p} \cX^\circ_{K_{hs}K^p}$ is the analytic sheafification of the presheaf on $\Perf$
    \[S = \Spa(R,R^+) \mapsto \{(S^\sharp, \Spf R^{\sharp +} \to \X)\}.\]
\end{lemma}
\begin{proof}
    The first part is \cite[2.1.4]{DHKZ24}. The second part follows also follows from loc.cit., by noting that the diamond functor $(-)^\Diamond$ on formal schemes commutes with inverse limits.
\end{proof}

For a suitably small away-from-$p$ level $K^p$, we have the identification 
\[\cX = \cX_{K^p} \times_{\cX_{K_{hs}K^p}} \cX_{K_{hs}} = \underline{\Isom}_G(T_p(\cA_K),\underline{\Lambda}) \times_{\cX_{K_{hs}K^p}} \cX_{K_{hs}},\]
which after passing to good reduction loci gives us
\[\cX^\circ = \cX_{K^p}^\circ \times_{\cX^\circ_{K_{hs}K^p}} \cX_{K_{hs}}^\circ.\]

Now we can use that analytic sheafification commutes with fiber products to obtain the following corollary.
\begin{corollary}\label{grlocus_moduli_cor}
    $\cX^\circ$ is the analytic sheafification of the presheaf on $\Perf$
    \[S = \Spa(R,R^+) \mapsto \{(S^\sharp, y,\beta)\}\]
    where $S^\sharp = \Spa(R^\sharp, R^{\sharp +})$ is an untilt of $S$ over $\Spa E$, $y: \Spf R^{\sharp+} \to \X$ is a map of formal schemes which corresponds to a formal abelian scheme $\fA/R^{\sharp+}$ with $G$-structure and infinite away-from-$p$ level structure, and 
    \[\beta: T_p(\cA_S) \overset{\sim}{\to} \underline{\Lambda}_S\]
    is a trivialization of $T_p(\cA_S)$, where $\pi_S: \cA_S^\diamond \to S$ is the adic generic fiber of $\fA^\diamond/\Spa(R^+,R^+)$ and $T_p(\cA_S)$ denotes the Tate module
    \[\underline{\Hom}_{\underline{\Z}_p}(R^1\pi_{S*} \underline{\Z_p}, \underline{\Z_p}).\]
\end{corollary}

For passing from the hybrid space to the abelian-type Shimura variety later, we will also need the following lemma.

\begin{lemma}\label{delta_torsor}
    The natural map $\cX^\circ \to \cX_G^\circ$ is a pro-étale torsor with respect to the profinite group $\Delta := \invlim_{K} \Delta(K)$.
\end{lemma}
\begin{proof}
    Since $(\cX^\circ \to \cX_G^\circ) = \invlim_{K} (\cX_K^\circ \to \cX_{G,K}^\circ)$, where the inverse limit runs over all uniformly good level structures $K \subset G(\A_f)$, we are reduced to checking that for such level subgroups, $\cX_K^\circ \to \cX_{G,K}^\circ$ is a finite étale torsor with respect to the finite group $\Delta(K)$.

    Firstly, $\M_{K_{hs}K^p} \to \M_{G,K_{hs}K^p}$ is a finite étale $\Delta(K_{hs}K^p)$-torsor because $K$ is \textit{uniformly} good. In fact, this is the key point where the uniformity condition is really necessary. Since $p$-adic completion, taking adic generic and taking diamonds commute with finite étale torsors (for the latter statement, see e.g. \cite[15.6]{Sch22}), we find that
    \[\cX_{K_{hs}K^p}^\circ \to \cX_{G,K_{hs}K^p}^\circ\]
    is a finite étale torsor with respect to $\Delta(K_{hs}K^p)$.

    Similarly, $\M_K \to \M_{G,K}$ is a finite étale $\Delta(K)$-torsor, and upon taking adic generic and taking diamonds, $\cX_K \to \cX_{G,K}$ is a finite étale $\Delta(K)$-torsor. Finally, consider the following commutative cube.

\[\begin{tikzcd}
	& {\cX_K} && {\cX_{G,K}} \\
	& {\cX_{K_{hs}K^p}} && {\cX_{G,K_{hs}K^p}} \\
	{\cX_K^\circ} && {\cX_{G,K}^\circ} \\
	{\cX_{K_{hs}K^p}^\circ} && {\cX_{G,K_{hs}K^p}^\circ}
	\arrow["{\Delta(K)}", from=1-2, to=1-4]
	\arrow[from=1-2, to=2-2]
	\arrow[from=1-4, to=2-4]
	\arrow["{\Delta(K_{hs}K^p)}", from=2-2, to=2-4]
	\arrow[from=3-1, to=1-2]
	\arrow[from=3-1, to=3-3]
	\arrow[from=3-1, to=4-1]
	\arrow[from=3-3, to=1-4]
	\arrow[from=3-3, to=4-3]
	\arrow[from=4-1, to=2-2]
	\arrow["{\Delta(K_{hs}K^p)}"', from=4-1, to=4-3]
	\arrow[from=4-3, to=2-4]
\end{tikzcd}\]
The left and right vertical faces of the cube are Cartesian by definition. The bottom of the cube is Cartesian because it consists of two opposite maps that are finite étale torsors with respect to the same group. This formally implies that the top of the cube is also Cartesian, and hence that $\cX_K^\circ \to \cX_{G,K}^\circ$ is a finite étale torsor with respect to $\Delta(K)$ as desired.
\end{proof}

\section{$p$-divisible groups}
In this section we review the theory of $p$-divisible groups, their classification over various base rings via Dieudonné theory and their connection to deformations of abelian varieties via Serre-Tate theory. This will be of vital importance for us; just as in \cite{Zha23}, the fiber product diagram for the hybrid space will end up being a geometric reformulation of Dieudonné and Serre-Tate theory.

\subsection{Definition and first properties}
Let $S$ be any scheme.

\begin{definition}
    A \textbf{$p$-divisible group} over $S$ is an fpqc-sheaf of abelian groups that can be written as an inductive system
    \[\cG = \dirlim G_n\] 
    indexed over $n \in \N$, where all $G_n$ are finite locally free group schemes over $S$, such that there is a natural number $h$ such that
    \begin{enumerate}[label=(\roman*)]
        \item for each $n$, $G_n$ has degree $p^{hn}$ over $S$.
        \item for each $n$ we have an exact sequence
        \[0 \to G_n \overset{i_n}{\longrightarrow} G_{n+1} \overset{p^n}{\longrightarrow} G_n \to 0\]
    \end{enumerate}
    The natural number $h$ is called the \textbf{height} of the $p$-divisible group $\cG$.
\end{definition}
Note that there are canonical isomorphisms $\cG[p^n] \simeq G_n$. From the definition we see that a $p$-divisible group $\cG$ is $p^\infty$-torsion and that multiplication by $p$ is surjective. 

Given a $p$-divisible group $\cG$, its \textbf{dual $p$-divisible group} $\cG^\vee$ is the fpqc-sheaf 
\[T \mapsto \dirlim \cG[p^n]^\vee (T),\]
where $\cG[p^n]^\vee$ is the Cartier dual of $\cG[p^n]$, with transition maps given by the duals of multiplication by $p$. Clearly $\cG^\vee$ is again a $p$-divisible group. 

\begin{definition}
    An \textbf{isogeny} $\lambda: \cG \to \cG'$ between $p$-divisible groups over $S$ is a surjection of fpqc-sheaves whose kernel is representable by a finite locally free group scheme. We will write $\underline{\Hom}(\cG, \cG')$ for the fpqc-sheaf of isogenies between $\cG$ and $\cG'$.

    A \textbf{quasi-isogeny} is a global section $\rho$ of the sheaf $\underline{\Hom}(\cG,\cG') \otimes \Q$, such that Zariski-locally on $S$, $p^n \rho$ is an isogeny for some integer $n$.
\end{definition}

\begin{definition}
    Let $\cG$ be a $p$-divisible group. A \textbf{polarization} of $\cG$ is a quasi-isogeny $\lambda: \cG \to \cG^\vee$ such that the Cartier dual $\lambda^\vee: \cG \iso \cG^{\vee\vee} \to \cG^\vee$ is equal to $-\lambda$. It is called a \textbf{principal polarization} if $\lambda$ is an isomorphism. 
\end{definition}

\begin{example}
    \begin{enumerate}[label=(\roman*)]
        \item The sheaf $\underline{\Q_p/\Z_p}_S$ is a $p$-divisible group of height 1. Its dual is the $p$-divisible group
        \[\mu_{p^\infty} := \dirlim \mu_{p^n},\]
        which is again of height one.
        \item If $A$ is an abelian scheme over $S$ of relative dimension $d$, then
        \[A[p^\infty] := \dirlim_n A[p^n]\]
        is a $p$-divisible group of height $2d$. Its dual is the $p$-divisible group $A^\vee[p^\infty]$.
    \end{enumerate}
\end{example}
The pairing between $A[p^\infty]$ and $A^\vee[p^\infty]$ (or rather their $p^n$-torsions) is precisely the Weil pairing. When $A$ is principally polarized, then $A[p^\infty]$ is self-dual via this principal polarization. In general a polarization $\lambda$ on $A$ induces a polarization of its $p$-divisible group $A[p^\infty]$. 

\begin{definition}
    The (integral) \textbf{Tate module} of a $p$-divisible group $\cG/S$ is the fpqc-sheaf $T_p\cG := \invlim_n \cG[p^n]$, where the transition maps are given by multiplication by $p$. It is a sheaf of $\Z_p$-modules and can be identified with the Hom-sheaf in the category of sheaves of abelian groups on $S_{\text{fpqc}}$
    \[T_p\cG \simeq \cH om(\Q_p/\Z_p,\cG).\]
\end{definition}

Since $T_p\cG$ is an inverse limit of schemes affine over $S$, it is representable by a scheme that is affine over $S$ and additionally flat over $S$, since all $\cG[p^n]$ are flat over $S$.

We will mostly be interested in $p$-divisible groups $\cG$ that are defined over a $p$-adically complete $\Z_p$-algebra $R$. In this case, we may intepret $\cG$ as an fpqc-sheaf on $\Nilp_R^{\opp}$, which is the opposite category of $R$-algebras on which $p$ is nilpotent, as follows:
\[\Nilp_R^{\opp} \ni A \mapsto \invlim_i\text{ } \dirlim_n \cG[p^n](A/p^i).\]

\begin{definition}
    The \textbf{formal completion} $\hat{\cG}$ of a $p$-divisible group $\cG$ over a $p$-adically complete $\Z_p$-algebra $R$ is the fpqc-sheaf on $\Nilp_R^{\opp}$ 
    \[A \mapsto \dirlim_k \{x \in \cG(A) : x = e_\cG \text{ in } \cG(A/I) \text{, for an ideal } I \subset A \text{ such that } I^{k+1}=0 \}\]
\end{definition}

\begin{prop}[\cite{SW13}, 3.1.2]
    $\hat{\cG}$ is a formal Lie (group) variety, and it is hence represented by an affine formal scheme over $R$, which Zariski locally on $R$ is isomorphic to
    \[\Spf R[[X_1,...X_d]]\]
    for some integer $d \geq 0$, called the dimension of $\cG$ relative to $R$.
\end{prop}

\begin{definition}
    The \textbf{Lie algebra} of $\cG$ is the fpqc-sheaf of $R$-modules $Lie(\cG) := Lie(\hat{\cG})$, defined to be the dual of the (Zariski) locally free $R$-module of rank $d$
    \[\omega_\cG := e_{\hat{\cG}}^* \Omega^1_{\hat{\cG}/R}.\]
    By $\Lie \cG$ we will denote its global sections, which is a finite locally free $R$-module. 
\end{definition}

Many of the $p$-divisible groups that we will encounter will have extra structure which is coming from abelian varieties with $G$-structure.

\begin{definition}\label{p-div_with_structure}
    Let $S$ be an $\O_E$-scheme, then we call a triple $(\cG,\iota,\lambda)$ a \textbf{$p$-divisible group with $G$-structure} over $S$ if
    \begin{itemize}
        \item $\cG$ is a $p$-divisible group over $S$,
        \item $\iota: \O_B \otimes \Z_p \to \End(\cG)$ is a $\Z_p$-linear map that satifies the Kottwitz determinant condition
        \[\det_{\O_S}(\iota(a) | \Lie(\cG)) = \det(a|V_{1,\Q_p})\]
        for any $a \in \O_B$, and
        \item $\lambda: \cG \to \cG^\vee$ is an $\O_p$-linear polarization.
    \end{itemize}
    A \textbf{$\Z_p$-isomorphism} resp. \textbf{$\O_p$-isomorphism} is an $\O_B$-linear isomorphism $f: \cG \to \cG'$ preserving the polarizations up to a section in $\underline{\Z_p^\times}(S)$ resp. $\underline{\O_p^\times}(S)$. A \textbf{$\Q_p$-quasi-isogeny} resp. \textbf{$F_p$-quasi-isogeny} is an $\O_B$-linear quasi-isogeny $f: \cG \to \cG'$ preserving the polarizations up to a section in $\underline{\Q_p^\times}(S)$ resp. $\underline{F_p^\times}(S)$.
\end{definition}

Note that if $\cG$ is a $p$-divisible group over $S$ with $G$-structure, then $T_p\cG$ is naturally an $\O_B$-module.

\subsection{Classification over $\O_C$}
Let $C$ be a complete non-archimedean algebraically closed field with ring of integers $\O_C$. Here we recall Scholze-Weinstein's classification of $p$-divisible groups over $\O_C$ in terms of linear algebraic data, similar in spirit to Riemann's classification of complex abelian varieties in terms of lattices and Riemann forms.

\begin{theorem}[\cite{SW20}, 12.1.1]
    Let $\cG$ be a $p$-divisible group over $\O_C$. There is a natural short exact sequence
    \[0 \to \Lie \cG \otimes_{\O_C} C(1) \overset{\alpha_{G^*}^*(1)} {\longrightarrow} T_p\cG(\O_C) \otimes_{\Z_p}C \overset{\alpha_G}{\longrightarrow} (\Lie \cG^\vee)^* \otimes_{\O_C} C \to 0,\]
    called the \textbf{Hodge-Tate filtration}.
\end{theorem}
We recall how to construct the map $\alpha_G$. Given a section $f$ of $T_p\cG$, we may view it as a homomorphism $f: \Q_p/\Z_p \to \cG$ and apply the Lie algebra functor to its dual $f^\vee: \cG^\vee \to \mu_{p^\infty}$ to obtain the map $\Lie(f^\vee): \Lie(\cG^\vee) \to \Lie(\mu_{p^\infty})$. After picking a coordinate $t$ of $\G_m$, the $\O_C$-linear dual $\Lie(\mu_{p^\infty})^*$ is naturally trivialized and isomorphic to $\O_C \frac{dt}{t}$, and we may define $\alpha_G$ via the assignment
\[f \mapsto (\Lie f^\vee)^*(\frac{dt}{t}).\]

Next, let $\{(T,W)\}$ be the category of pairs consisting of a finite free $\Z_p$-module $T$ and a $C$-sub-vector space $W \subset T \otimes_{\Z_p} C(-1)$, with the obvious notion of morphisms. We define the dual of the pair $(T,W)$ to be the pair $(T^*(1),W^\perp)$, where the orthogonal complement is with respect to the natural pairing between $T \otimes_{\Z_p} C(-1)$ and $T^* \otimes_{\Z_p} C$. Then the main classification result is the following.

\begin{theorem}[\cite{SW13}, Theorem B, 5.2.1]\label{sw_pdiv}
    The category of $p$-divisible groups over $\O_C$ is equivalent to the above category $\{(T,W)\}$ via the functor
    \[\Psi: \cG \mapsto (T_p\cG(\O_C),\Lie\cG \otimes_{\O_C} C),\]
    where $\Lie \cG \otimes_{\O_C} C$ is viewed as a $C$-sub-vector space of $T_p\cG(\O_C) \otimes_{\Z_p} C(-1)$ via the Hodge-Tate filtration. This equivalence is compatible with duality. 
\end{theorem}
\begin{proof}
    For convenience, we recall the construction of the inverse functor in the case that $C$ is spherically complete (i.e. the intersection of every decreasing sequence of balls is non-empty) and the norm map $C \to \R_{\geq0}$ is surjective.

    In this case, given a tuple $(T,W)$, we define the $p$-divisible group
    \[\cG' := T(-1) \otimes_{\Z_p} \mu_{p^\infty}\]
    and obtain the following diagram
\[\begin{tikzcd}
	& {W \otimes_C \mathbb{G}_a} \\
	{\cG'^{\ad}_\eta} & {T(-1) \otimes_{\Z_p} \mathbb{G}_a,}
	\arrow[hook, from=1-2, to=2-2]
	\arrow[from=2-1, to=2-2]
\end{tikzcd}\]
    where $\G_a$ is the sheafification of the functor $(A,A^+) \mapsto A^+$ on complete affinoid $(C,\O_C)$-algebras. The vertical arrow is induced by the given inclusion $W \hookrightarrow T \otimes_{\Z_p} C(-1)$, while the horizontal arrow is the logarithm on $\cG_\eta'^{\ad}$. Note that the logarithm usually takes the form $\cG'^{\ad}_\eta \to \Lie\cG' \otimes_{\O_C} \G_a$, but since $\Lie \mu_{p^\infty} \iso \O_C$, $\Lie \cG'$ is naturally in $T$ identified with $T(-1) \otimes_{\O_C} C$.

    Denote by $\cG^{\ad}_\eta$ the sheaf-theoretic fiber product of the above diagram. Using results from \cite{Far19}, one can show under our additional assumptions on $C$ that the formal scheme
    \[\cG := \coprod_Y \Spf H^0(Y,\O_Y^+)\]
    is a $p$-divisible group over $\O_C$, where $Y$ runs over all connected components of $\cG_\eta^{\ad}$. This defines the inverse functor $(T,W) \mapsto \cG$.
\end{proof}

\subsection{Descent properties}
Similar to the approach in \cite{Zha23}, we will need certain descent properties for $p$-divisible groups in order to obtain the fiber product diagram for our hybrid spaces. We collect the necessary results from \cite[Section 3.3]{Zha23}.

\begin{lemma}[\cite{Zha23}, Lemma 3.14]\label{milnor}
    Given a so-called Milnor square, i.e. a Cartesian diagram of rings
\[\begin{tikzcd}
	R & {R_2} \\
	{R_1} & {R_3}
	\arrow[from=1-1, to=1-2]
	\arrow[from=1-1, to=2-1]
	\arrow["\ulcorner"{anchor=center, pos=0.125}, draw=none, from=1-1, to=2-2]
	\arrow[from=1-2, to=2-2]
	\arrow[two heads, from=2-1, to=2-2]
\end{tikzcd}\]
with $R_1 \twoheadrightarrow R_3$ surjective, the corresponding diagram of categories of finite locally free modules over these rings is $2$-Cartesian. Concretely, the category of finite locally free $R$-modules is equivalent to the category of gluing data
\[\{M_1/R_1, M_2/R_2, \alpha: M_1 \otimes_{R_1} R_3 \iso M_2 \otimes_{R_2} R_3\},\]
where $M_i$ is a finite locally free module over $R_i$, and $\alpha$ is an isomorphism.
\end{lemma}
\begin{proof}
    Given a finite locally free $R$-module, we can base change it to $R_1$ and $R_2$ and take $\alpha$ to be the identity. Conversely, given a gluing triple $(M_1,M_2,\alpha)$ we obtain an $R$-module $M$ via
    \[M := \ker(M_1 \oplus M_2 \overset{\alpha - \id}{\longrightarrow} M_2 \otimes_{R_2} R_3).\qedhere\]
\end{proof}

\begin{prop}[\cite{Zha23}, 3.16]\label{BT_milnor}
    Given a Milnor square as in Lemma \ref{milnor}, the corresponding diagram of categories of $p$-divisible groups over these rings is $2$-Cartesian. 
\end{prop}
\begin{proof}
    Given two $p$-divisible groups $\cG,\cG$ over $R$, we have
    \[\Hom_R(\cG,\cG') = \invlim_n \Hom_R(\cG[p^n],\cG'[p^n]).\]
    For each $n$, $\Hom_R(\cG[p^n], \cG'[p^n])$ is given by maps of $R$-modules $\O(\cG'[p^n]) \to \O(\cG[p^n])$ that respect the additional Hopf algebra structure. Since all the Hopf algebra structure is given in terms of morphisms of $R$-modules, Lemma \ref{milnor} shows that giving a map in $\Hom_R(\cG[p^n], \cG'[p^n])$ is equivalent to giving morphisms on the base changes to $R_1$ and $R_2$ that agree over $R_3$. After passing to the inverse limit, this shows fully faithfulness.

    For essential surjectivity, we argue as follows. Given a gluing triple $(\cG_1,\cG_2,\alpha)$ of $p$-divisible groups, we can first restrict to $p^n$-torsion points for each $n$ to obtain the ring of regular functions $\O(\cG[p^n])$ by Lemma \ref{milnor} and endow it with a Hopf algebra structure. Fully faithfulness shows that there are transition maps turning $\invlim_n \cG[p^n]$ into a $p$-divisible group that restricts to $\cG_1$ resp. $\cG_2$.
\end{proof}

\begin{lemma}[\cite{Zha23}, 3.17]\label{prod_lemma}
    Let $R= \prod_{i \in I} V_i$ be a product of valuation rings, and let $n$ be an integer. Then the category of rank $n$ projective $R$-modules is equivalent to the collection of categories of rank $n$ projective $V_i$-modules.
\end{lemma}
\begin{proof}
    Any rank $n$ projective $R$-module yields a collection of rank $n$ projective $V_i$-modules via base change. Conversely, given a collection $\{M_i\}_{i \in I}$ of rank $n$ projective (also automatically free) $V_i$-modules, we can set $M := \prod_{i \in I} M_i$, which is also free. These two functors are clearly inverse to each other. 
\end{proof}

\begin{corollary}[\cite{Zha23}, 3.18]\label{prod_cor}
    In the setting of Lemma \ref{prod_lemma}, the category of $p$-divisible groups over $R$ of a fixed height is equivalent to the collection of $p$-divisible groups of the same height over each $V_i$.
\end{corollary}

\subsection{Dieudonné modules}
In this subsection we recall the necessary background in Dieudonné theory. We will focus on prismatic Dieudonné theory as developed in \cite{AB23}, since all classification results we will actually need are subsumed in this theory.

The base rings of interest are so-called \textbf{quasi-syntomic} rings. These are $p$-complete rings with bounded $p^\infty$-torsion such that their cotangent complex $L_{R/\Z_p}$ has $p$-complete Tor-amplitude in $[-1,0]$, i.e. the complex 
\[M = L_{R/\Z_p} \otimes^{\mathbb{L}}_R R/p \in D(R/p)\]
satisfies that $M \otimes_R^{\mathbb{L}} N \in D^{[-1,0]}(R/p)$ for any $R/p$-module $N$ (see \cite[3.3.1]{AB23}).

In \cite{AB23}, Anschütz and Le Bras attach to a $p$-divisible group $\cG$ over a quasi-syntomic ring $R$ its prismatic Dieudonné crystal as a sheaf on the small quasi-syntomic site of $R$, which has as covers $p$-complete faithfully flat maps between quasi-syntomic rings. Much simplification arises in the following special case of quasi-syntomic rings. 

\begin{definition}[\cite{AB23}, 3.3.5]
    A quasi-syntomic ring $R$ is called \textbf{quasi-regular semiperfectoid} if there exists a surjection $S \twoheadrightarrow R$ from an integral perfectoid ring $S$ (see Definition \ref{int_perf} below).
\end{definition}

\begin{example}[\cite{AB23}, 3.3.6]
    Any integral perfectoid ring $R$ and any $p$-complete quotient of $R$ by a finite regular sequence, with bounded $p^\infty$-torsion, is quasi-regular semiperfectoid. This includes the main two cases in which we will apply the theory of prismatic Dieudonné modules. Namely, if $(R,R^+)$ is a perfectoid Tate-Huber pair with pseudo-uniformizer $\varpi \in R^+$, then $R^+$ and $R^+/\varpi$ are quasi-regular semiperfectoid (see Lemma \ref{comparison_perfectoid} below).
\end{example}

If $R$ is quasi-regular semiperfectoid, then by \cite[7.2]{BS22} the initial prism $(\Prism_R,I)$ on the prismatic site $(R)_\Prism$ exists and by \cite[4.1.13]{AB23}, giving the prismatic Dieudonné crystal of a $p$-divisible group over $R$ is equivalent to evaluating its associated sheaf on the prismatic site $(R)_\Prism$ at the initial prism. $(\Prism_R,I)$. This process yields the so-called \textbf{contravariant prismatic Dieudonné module}, a finite locally free $\Prism_R$-module $M$ that is equipped with an endomorphism $\varphi_M: \varphi^*M \to M$ (where $\varphi$ is the Frobenius on $\Prism_R$). 

The above construction is contravariant in $\cG$. We will instead work with the \textbf{covariant} version, meaning that we will apply $\Hom_{\Prism_R}(-,\Prism_R)$ to obtain the \textbf{covariant prismatic Dieudonné module} $M_\Prism(\cG)$. The main classification is then the following.

\begin{theorem}[\cite{AB23}, 4.1.12, 4.6.10,]
    Let $R$ be a quasi-regular semi-perfectoid ring. Then the prismatic Dieudonné functor, sending a $p$-divisible group $\cG$ over $R$ to its (covariant) prismatic Dieudonné module $M_\Prism(G)$, is an equivalence between the categories of $p$-divisible groups over $R$ and admissible Dieudonné modules over $R$.
\end{theorem}

As in \cite{Zha23}, the following two examples will be our key applications of prismatic Dieudonné theory. 

\begin{example}[\cite{AB23}, 4.3.6]
    If $R$ is an integral perfectoid ring, the associated prism is
    \[(\Prism_R, I) = (W(R^\flat), \ker(\theta \circ \varphi^{-1})),\]
    where $\theta: W(R^\flat) \to R$ is Fontaine's theta map. In this case an admissible prismatic Dieudonné module is the same as a \textbf{minuscule Breuil-Kisin-Fargues module} with one leg at $V(\ker(\theta \circ \varphi^{-1}))$. That is, it is a finite projective $W(R^\flat)$-module $M$ together with a $\varphi$-linear isomorphism
    \[\varphi_M: M[\frac{1}{\xi}] \overset{\sim}{\longrightarrow} M[\frac{1}{\varphi(\xi)}],\]
    where $\xi$ is a chosen generator of $\ker(\theta \circ \varphi_R^{-1})$, such that $M \subset \varphi_M(M) \subset \frac{1}{\varphi(\xi)}M$.
\end{example}

\begin{example}[\cite{AB23}, 4.3.3]
    Let $R$ be quasi-regular semi-perfectoid such that $pR = 0$, e.g. $R = R^+/\varpi$ where $R^+$ is integral perfectoid and $\varpi$ a pseudo-uniformizer. Then its associated prism is
    \[(\Prism_R, I) = (A_{\cris}(R), (p)),\]
    and in this case an admissible prismatic Dieudonné module agress with the naive dual of the contravariant crystalline Dieudonné module of Berthelot, Breen and Messing.
\end{example}

For any integral perfectoid ring $R^+$ in characteristic $p$, there is a natural map $W(R^+) \to A_{\cris}(R^+/\varpi)$, and for any $p$-divisible group $\cG$ over $R^{\sharp+}$ there is a comparison isomorphism
\begin{align}
    M_\Prism(\cG) \otimes_{W(R^+)} A_{\cris}(R^+/\varpi) \overset{\sim}{\to} M_{\cris}(\cG \times_{R^{\sharp+}} R^+/\varpi),
\end{align}
see \cite[17.5.2]{SW20}. Later on, this comparison isomorphism is going to be the reason why the diagram we want to show is Cartesian is 2-commutative.

\subsection{Serre-Tate theory}
In general, Serre-Tate theory answers the question inhowfar the deformations of an abelian variety are controlled by its associated $p$-divisible group. We include two such results that will be needed later on, in particular to show that the definition of our hybrid Igusa stack is well-defined.

\begin{theorem}[\cite{CS24}, 2.4.1]\label{st_1}
    Let $S' \twoheadrightarrow S$ be a surjection of rings in which $p$ is nilpotent, such that its kernel $I \subset S'$ is nilpotent.
    \begin{enumerate}
        \item The functor $\cG_{S'} \mapsto \cG_{S'} \times_{S'} S$ from the category of $p$-divisible groups up to isogeny over $S'$ to the category of $p$-divisible groups up to isogeny over $S$ is an equivalence.
        \item The functor $A_{S'} \mapsto A_{S'} \times_{S'} S$ from the category of abelian schemes up to $p$-power isogeny over $S'$ to the category of abelian schemes up to $p$-power isogeny over $S$ is an equivalence. 
    \end{enumerate}
\end{theorem}

\begin{theorem}[\cite{CS24}, 2.4.2]\label{st_2}
    Let $S' \twoheadrightarrow S$ be a surjection of rings in which $p$ is nilpotent, such that its kernel $I \subset S'$ is nilpotent. Then the functor
    \[A_{S'} \mapsto (A_S, A_{S'}[p^\infty],\id)\]
    is an equivalence between the category of abelian schemes over $S'$ and the category of triples consisting of an abelian scheme $A_S$ over $S$, a $p$-divisible group $\cG_{S'}$ over $S'$ and an isomorphism $\rho: A_S[p^\infty] \overset{\sim}{\to} \cG_{S'} \times_{S'} S$.
\end{theorem}
\begin{proof}
    Let $i: \Spec(S) \hookrightarrow \Spec(S')$ denote the inclusion and fix $N \gg 0$ such that $(p,I)^N = 0$ both in $S'$.

    First, we prove fully faithfulness. Assume we are given two abelian schemes $A_0, B_0$ over $S$ with lifts $A, B$ over $S'$. We will view all these abelian schemes as fpqc-sheaves over their respective base. Denote by $K_B$ the kernel of the reduction map $B \to i_*B_0$. Applying $\Hom_{S'}(A,-)$, we get the exact sequence
    \[\Hom_{S'}(A,K_B) \to \Hom_{S'}(A,B) \overset{res}{\to} \Hom_S(A_0, B_0) \to \Ext_{S'}^1(A,K_B).\]
    In order to prove fully faithfulness, we analyze $\Hom_{S'}(A,K_B)$ and $\Ext_{S'}^1(A,K_B)$. For this, we consider the multiplication-by-$p^N$ short exact sequence
    \begin{align*}
        0 \to A[p^N] \to A \overset{p^N}{\to} A \to 0
    \end{align*}
    Applying $\Hom_{S'}(-,K_B)$ to the above, we get the long exact sequence
    \begin{align*}
        0 \to \Hom_{S'}(A,K_B) &\overset{p^N}{\to} \Hom_{S'}(A,K_B) \to \Hom_{S'}(A[p^N],K_B)\\
        & \to \Ext_{S'}^1(A,K_B) \overset{p^N}{\to} \Ext_{S'}^1(A,K_B) 
    \end{align*}
    Since $K_B$ is killed by $p^N$, the two arrows labelled by $p^N$ are zero. First, this implies that $\Hom_{S'}(A,K_B) = 0$ and thus the functor is faithful. Second, we see that $\Hom_{S'}(A[p^N],K_B) \overset{\sim}{\to} \Ext_{S'}^1(A,K_B)$, which implies that for any map $f_0 \in \Hom_{S}(A_0,B_0)$, $p^Nf_0$ is always liftable to some $f': A \to B$. When applying $\Hom_{S'}(-,B)$ to the multplication-by-$p^N$ sequence, one sees that $f_0$ itself is liftable if $f'$ annihilates $A[p^N]$. To finish the proof of fullness, note that we are additionally given a map $f[p^\infty]: A[p^\infty] \to B[p^\infty]$ that lifts $f_0[p^\infty]$. Thus by part (1) of Theorem \ref{st_1}, it is necessary that $f'$ induces $p^Nf[p^\infty]: A[p^\infty] \to B[p^\infty]$. In particular, $f'$ annihilates $A[p^N]$ as desired. 

    Next, we show essential surjectivity. Assume we are given an abelian scheme $A_0$ over $S$ together with a $p$-divisible group $\cG$ over $S'$ and an isomorphism $A_0[p^\infty] \iso \cG \times_{S'} S$. According to Theorem \ref{st_1}, we may pick an arbitrary lift $A'$ of $A_0$ over $S$ (which is unique up to $p$-power isogeny). Without loss of generality, we may assume that the map $A'_S \to A_0$ is an actual isogeny (instead of just a quasi-isogeny). By Theorem \ref{st_1}, the induced isogeny $A'_S[p^\infty] \to A_0[p^\infty]$ lifts to an isogeny $\varphi: A'[p^\infty] \to \cG$ over $S'$. We define $A := A'/\ker\varphi$, and then $A[p^\infty] \iso \cG$ and $A$ still lifts $A_0$.
\end{proof}

\section{Mixed-characteristic affine Grassmannian and Hodge-Tate period map}
In this chapter we will introduce the $B_{\dR}^+$-affine Grassmannian and extend the construction of the Hodge-Tate period map from Hodge-type Shimura varieties to our hybrid space and our abelian-type Shimura variety.

\subsection{Mixed-characteristic affine Grassmannian}
Since the content in this subsection holds in much greater generality, we will fix a more general framework for the time being. Namely, we fix a finite extension $E/\Q_p$ with ring of integers $\O_E$, uniformizer $\pi$ and residue field $\F_q$. Furthermore, we fix a connected reductive group $G$ over $E$. In the next subsection, we will revert back to our usual notation.

Consider the functor of \textbf{ramified Witt vectors} 
\[W_{\O_E}(-): \{\text{Perfect $\F_q$-algebras}\} \longrightarrow \{\text{$\pi$-torsionfree $\pi$-complete $\O_E$-algebras}\},\]
which via the $p$-typical Witt vectors $W(-)$ is given by $W_{\O_E}(R) = W(R) \widehat{\otimes}_{W(\F_q)} \O_E$. For a perfect $\F_q$-algebra $R$ with $q$-Frobenius $\varphi$, we write $\varphi_R := W_{\O_E}(\varphi)$ for the unique lift of $\varphi$ to $W_{\O_E}(R)$ and call it the Frobenius endomorphism on $W_{\O_E}(R)$.

$W_{\O_E}(-)$ is left-adjoint to the \textbf{tilting functor}
\[(-)^\flat: A \mapsto \invlim_{x \mapsto x^q} A/p.\]
The adjunction counit $\theta: W_{\O_E}(A^\flat) \to A$ is called \textbf{Fontaine's theta map}. 

\begin{definition}[\cite{SW20} 17.5.1]\label{int_perf}
    An $\O_E$-algebra $A$ is called \textbf{integral perfectoid} if it is of the form $W_{\O_E}(R)/I$ for some perfect $\F_q$-algebra $R$ and a principal ideal $I$, such that $W(R)$ is $I$-adically complete and such that $I$ is generated by a \textbf{distinguished element} $d$, i.e. an element satisfying
    \[\frac{\varphi_R(d) - d^p}{p} \in W_{\O_E}(R)^\times\]
\end{definition}
Given an integral perfectoid $O_E$-algebra $A = W_{\O_E}(R)/I$, we have a canonical isomorphism $R \simeq A^\flat$ which identifies $I$ with $\ker(\theta)$. Furthermore, any generator of $I$ turns out to be distinguished. We will denote such a choice of generator by $\xi$.

The concept of integral perfectoid algebras relates to the notion of perfectoid Tate-Huber pairs (see Definition \ref{thp}) as follows.

\begin{lemma}[\cite{BMS18}, 3.21, 3.22]\label{comparison_perfectoid}
    If $(R,R^+)$ is a perfectoid Tate-Huber pair, then $R^+$ is an integral perfectoid $\Z_p$-algebra. Conversely if $A$ is an integral perfectoid $\Z_p$-algebra that is $\varpi$-complete for some non-zerodivisor $\varpi$ such that $\varpi^p$ divides $p$, then $(A[1/\varpi], A)$ is a perfectoid Tate-Huber pair (endowed with the $\varpi$-adic topology). 
\end{lemma}

\begin{definition}
    Let $R$ be a perfect $\F_q$-algebra. An \textbf{untilt} of $R$ over $\O_E$ is a pair $(A,\iota)$, where $A$ is an integral perfectoid $\O_E$-algebra and $\iota: A^\flat \overset{\sim}{\to} R$ is a fixed isomorphism.
\end{definition}

We will usually suppress the isomorphism $\iota$ in our notation.

\begin{definition}
    Given an integral perfectoid $\O_E$-algebra $R^{\sharp +}$ with tilt $R^+$, let $\xi$ be a generator of $\ker(\theta)$ and assume that $R^\sharp := R^{\sharp +}[1/\pi] \neq 0$, then we define the following \textbf{de Rham period rings}: We denote by $B_{\dR}^+(R^\sharp)$ the $\xi$-adic completion of $W_{\O_E}(R^+)[1/\pi]$, and define $B_{\dR}(R^\sharp) := B_{\dR}^+(R^\sharp)[1/\xi]$.
\end{definition}

As the notation already suggests, these definitions only depend on the perfectoid Tate ring $R^\sharp$ and not on the integral subring $R^{\sharp +}$.

\begin{remark}
    If $C$ is an algebraically closed, complete non-archimedean field extension of $E$, then by the Cohen structure theorem we have an isomorphism
    \[B_{\dR}^+(C) \iso C[\![\xi]\!],\]
    which is however highly non-canonical (in particular it is not Galois equivariant).
\end{remark}

\begin{definition}[\cite{SW20}, 19.1] 
    The \textbf{$B_{\dR}^+$-affine Grassmannian} $\Gr_G$ associated to $G$ is the v-sheaf on $\Perf / \Spd E$ that sends $S = \Spa(R,R^+)$ with a fixed untilt $S^\sharp = \Spa(R^\sharp,R^{\sharp+}) / \Spa E$ to the isomorphism classes of tuples $(\cF,\alpha)$, where $\cF$ is a $G$-torsor over $\Spec B_{\dR}^+(R^\sharp)$ and $\alpha$ is a trivialization of $\cF$ over $\Spec B_{\dR}(R^\sharp)$.
\end{definition}
Equivalently, this is also the étale sheafification of the presheaf
\[S / \Spd E \longmapsto G(B_{\dR}(R^\sharp))/G(B_{\dR}^+(R^\sharp)),\]
see \cite[19.1.2]{SW20}.

As we will see later when introducing the relative Fargues-Fontaine curve $X_S$ (ranging over $S \in \Perf$), the v-sheaf $\Gr_G$ can be interpreted in terms of modifications of $G$-torsors on the Fargues-Fontaine curve. Namely, an $S$-point of $\Gr_G$ will correspond to a modification of the trivial $G$-torsor on the Fargues-Fontaine curve $X_S$ at the Cartier divisor associated with the untilt $S^\sharp$ of $S$.

The $B_{\dR}^+$-affine Grassmannian enjoys very nice geometric properties, some of which we collect below.

\begin{prop}[\cite{SW20}, 19.1.4,19.2.4]
    The $B_{\dR}^+$-affine Grassmannian $\Gr_G$ is partially proper and a union of spatial diamonds.
\end{prop}

\begin{prop}[\cite{SW20}, 19.1.5]
    Given a closed embedding $H \hookrightarrow G$ of reductive groups over $\Q_p$, the induced map $\Gr_H \hookrightarrow \Gr_G$ is also a closed embedding.
\end{prop}

Fixing a complete algebraically closed non-archimedean extension $C/E$ and fixing a split torus and Borel $T \subset B \subset G_C$, we have the \textbf{Cartan decomposition} indexed by the dominant cocharacters $X_*^+(T)$ of $T$ (i.e. those cocharacters $\mu \in X_*(T)$ that pair non-negatively with all positive roots that are determined by the choice of $B$):
\[G(B_{\dR}(C)) = \coprod_{\mu \in X_*^+(T)} G(B_{\dR}^+(C)) \cdot \xi^\mu \cdot G(B_{\dR}^+(C)),\]
where $\xi^\mu$ is the image of $\mu(\xi) \in G(B_{\dR}(C))$. For example, this follows from the classical Cartan decomposition by choosing a (non-canonical) isomorphism $B_{\dR}(C) \iso C(\!(\xi)\!)$.

Recall the \textbf{Bruhat order} on $X_*^+(T)$: given $\lambda,\mu \in X_*^+(T)$, we say that $\lambda \leq \mu$ if and only if $\mu - \lambda$ is a non-negative integral linear combination of simple coroots. Using the Bruhat order, the Cartan decomposition yields the \textbf{Bruhat stratification} on $\Gr_{G,C}$, the base change of $\Gr_G$ to $\Spd C$, by defining the following two subfunctors:

\begin{definition}
    Let $\mu \in X_*^+(T)$ be dominant. We let $\Gr_{G,C,\mu}$ be the subfunctor of $\Gr_{G,C}$ that sends $S / \Spd C$ to the set of maps $S \to \Gr_{G,C}$ such that for every geometric point $x = \Spa(\tilde{C}, \tilde{C}^+)$ of $S$, the corresponding $\Spa(\tilde{C}, \tilde{C}^+)$-point of $\Gr_{G,C}$ lies in the coset 
    \[G(B_{\dR}^+(\tilde{C})) \cdot \xi^\mu \cdot G(B_{\dR}^+(\tilde{C}))\]
    of the Cartan decomposition. 

    We define the subfunctor $\Gr_{G,C,\leq \mu}$ of $\Gr_{G,C}$ the same way, except that we allow the corresponding $\Spa(\tilde{C}, \tilde{C}^+)$-point to lie in 
    \[\coprod_{\lambda \leq \mu} G(B_{\dR}^+(\tilde{C})) \cdot \xi^\lambda \cdot G(B_{\dR}^+(\tilde{C})),\]
    where $\leq$ denotes the Bruhat order on $X_*^+(T)$.
\end{definition}

Note that if the $G(C)$-conjugacy class of $\mu$ is already defined over $E$, then $\Gr_{G,\mu}$ is also defined over $E$. Recall that given a cocharacter $\mu \in X_*(T)$, its associated maximal parabolic $P_\mu$ of $G$ is the subgroup of elements $g \in G$ such that the limit $\lim_{t \to 0} \mu(t)^{-1}g\mu(t)$ exists.

\begin{prop}[\cite{SW20}, 19.4.2]
    For any dominant cocharacter $\mu \in X_*^+(T)$, there is a natural Bialynicki-Birula map
    \[\pi_\mu: \Gr_{G,\mu} \to \Fl_{G,\mu}^\diamond,\]
    where $\Fl_{G,\mu}$ is the analytified partial flag variety $G/P_\mu$, where $P_\mu$ is the maximal parabolic subgroup associated with $\mu$. 

    If $\mu$ is minuscule, the Bialynicki-Birula map is an isomorphism.
\end{prop}

\subsection{Explicit computations}\label{grass_explicit}
Let us now return to the setup where $G$ is the reductive group over $\Q$ (or base changed to $\Q_p$) coming from our PEL datum, and where $E$ is the localization of the reflex field $E_0$ at the fixed place $v$ above $p$. To ease notation we will write $\Gr_G = \Gr_{G_{\Q_p}}$ in the following.

Let $\mu$ denote the inverse of the Hodge cocharacter associated to our Shimura datum, which is by design minuscule. To ease notation further, we will abbreviate the Schubert cell labelled by $\mu$ as $\Gr := \Gr_{G,\mu}$.

Let $C$ be an algebraically closed, non-archimedean complete extension of $E$. We will need a description of $\Spd C$-valued points of $\Gr$ in terms of $p$-divisible groups. Since $\Gr$ is partially proper, this is equivalent to describing more general $\Spd(C,C^+)$-valued points of $\Gr$. The Bialynicki-Birula isomorphism $\Gr \simeq \Fl_{G,\mu}^\diamond$ leads us to describe $\Spd C$-valued points of the analytified partial flag variety instead.

\begin{lemma}\label{pts_of_gr}
    $\Spd C$-points of $\Gr$ are in bijection with $p$-divisible groups $\cG$ over $\O_C$ with $G$-structure (see Definition \ref{p-div_with_structure}), together with a trivialization $\alpha: T_p\cG \iso \Lambda$  that is $\O_B$-linear and is an $\O_p$-symplectic similitude with respect to the (lifted) Weil pairing on $T_p\cG$ and the pairing $\langle\cdot,\cdot\rangle_F$ on $\Lambda$.
\end{lemma}
\begin{proof}
    Using the Bialynicki-Birula isomorphism, we need to describe $\Spd C$-points of $\Fl_{G,\mu}^{\diamond}$. Giving such a point is equivalent to giving an $\O_B$-invariant subspace 
    \[W \subset V \otimes_\Q C \]
    that is maximal isotropic with respect to the pairing $\langle\cdot,\cdot\rangle_F$. 

    By Theorem \ref{sw_pdiv}, the $\Z_p$-lattice $\Lambda$ (which is self-dual with respect to the symplectic pairing $\langle\cdot,\cdot\rangle_F$) and the subspace 
    \[W \subset \Lambda \otimes_{\Z_p} C(-1) = V \otimes_\Q C(-1)\]
    determine a $p$-divisible group $\cG/\O_C$ with trivialized Tate module $T_p\cG \iso \Lambda$. We will silently identify $V \otimes_\Q C(-1)$ with $V \otimes_\Q C$. The $p$-divisible groups arising this way come with the following extra structure:
    \begin{itemize}
        \item The $\O_B$-module structure on the tuple $(\Lambda,W)$ translates into an $\O_B \otimes \Z_p$-module structure on $\cG$. In particular, $T_p\cG$ is also naturally an $\O_B$-module and not just a $\Z_p$-module, and the isomorphism $T_p\cG \iso \Lambda$ is $\O_B$-linear. Note that the Kottwitz determinant condition is also satisfied.
        \item By self-duality with respect to $\langle\cdot,\cdot\rangle_F$, $\Lambda$ can be identified with its dual $\Lambda^*$ as $\O_B$-modules. This induces an $\O_B$-linear polarization $\lambda: \cG \to \cG^\vee$ via
        \[(\Lambda,W) \to (\Lambda^*(1),W^\bot), (t,w) \mapsto (\langle\cdot,t\rangle_F, \langle\cdot,w\rangle_F).\]
        \item The infinite level structure $T_p\cG \iso \Lambda$ is automatically an $\O_p$-symplectic similitude with respect to the (lifted) Weil pairing $\langle\cdot,\cdot\rangle_\lambda$ on $T_p\cG$ and $\langle\cdot,\cdot\rangle_F.$ on $\Lambda$.\qedhere
\end{itemize}
\end{proof}

\subsection{The Hodge-Tate period map}
In this subsection, we extend the Hodge-Tate period map of \cite{CS17} to the good reduction locus of our hybrid space and our abelian-type Shimura variety. We keep the notation from the last subsection. In particular, let $[\nu_h]$ be the $G(\C)$-conjugacy class of the Hodge cocharacter of the Shimura datum. We fix an isomorphism $\overline{\Q}_p \iso \C$ and a Borel pair $(T,B)$ of $G_{\overline{\Q}_p}$. By $\mu$ we will denote a dominant cocharacter that represents the $G(\overline{\Q}_p)$-conjugacy class of $\nu_h^{-1}$.

\begin{prop}\label{ht_hybrid}
    For any sufficiently small tame level subgroup $K^p \subset G(\A_f^p)$, there exists a Hodge-Tate period map of diamonds over $\Spd E$ 
    \[\pi_{HT}: \cX_{\alpha,K^p}^\circ \to \Gr_G\]
    with image lying in $\Gr$. It is equivariant with respect to Hecke action by $G(\Q_p)$, and $G(\A_f^p)$-invariant on the system $\{\cX_{\alpha,K^p}^\circ\}_{\alpha,K^p}$.
\end{prop}
\begin{proof}
    We follow the argument in \cite[6.14]{Zha23}. Note that it suffices to construct the Hodge-Tate period map on $\cX_{\alpha,K^p}^{\circ,\pre}$, which has a moduli description in terms of formal abelian schemes with $G$-structure together with a trivialization of the $p$-adic Tate module, see Corollary \ref{grlocus_moduli_cor}.

    Let $S = \Spa(R,R^+)$ and let $(S^\sharp,\fA,\beta)$ be an $S$-point of $\cX_{\alpha,K^p}^{\circ,\pre}$, i.e. $S^\sharp = \Spa(R^\sharp,R^{\sharp+})$ is an untilt of $S$ over $\Spa E$, $\fA$ is a $\Spf R^{\sharp+}$-point of $\X_{\alpha,K^p}$, and $\beta \in \underline{\Isom}_G(T_p\cA,\underline{\Lambda})(S)$ is a trivialization of the Tate module of the generic fiber $\cA$ of $\fA^\diamond$. Let $(M := M(\fA[p^\infty]), \varphi)$ be the prismatic Dieudonné module, which is a finite locally free $W(R^+)$-module equipped with an $\O_B$-module structure and a symplectic pairing $\langle\cdot,\cdot\rangle_M$.
    
    Consider the étale sheaf over $X := \Spec(B_{\dR}^+(R^{\sharp}))$
    \[\cF := \underline{\Isom}_G(M \otimes_{W(R^+)} \O_X, \Lambda \otimes_{\Z_p} \O_X)\]
    of $O_B$-linear isomorphisms that are $F_p$-symplectic similitudes with respect to the pairings $\langle\cdot,\cdot\rangle_M$ and $\langle\cdot,\cdot\rangle_F$. Since we can identify $G_{\Q_p}$ with the group of automorphisms of $\Lambda$ that are $F_p$-symplectic similitudes with respect to $\langle\cdot,\cdot\rangle_F$, this is a $G_{\Q_p}$-torsor.

    To see that $\cF$ is trivialized over $B_{\dR}(R^\sharp)$, first write $T = T_p\cA(S)$, which is a finite projective module over the ring $C^0(\Spec R^+,\Z_p)$. There is an étale-crystalline comparison isomorphism, whose base change to $B_{\dR}(R^\sharp)$ yields
    \[T \otimes_{C^0(\Spec R^+,\Z_p)} B_{\dR}(R^\sharp) \simeq M_{\crys}(\fA[p^\infty] \times_{R^{\sharp+}} R^{\sharp+}/p) \otimes_{A_{\crys}(R^+)} B_{\dR}(R^\sharp)\]
    By \cite[17.5.2]{SW20}, we also have the compatibility isomorphism
    \[M \otimes_{W(R^+)} A_{\crys}(R^+) \simeq M_{\crys}(\fA[p^\infty] \times_{R^{\sharp+}} R^{\sharp+}/p).\]
    Combining these two isomorphisms yields a natural isomorphism
    \[c: T \otimes_{C^0(\Spec R^+,\Z_p)} B_{\dR}(R^\sharp) \overset{\sim}{\to} M \otimes_{W(R^+)} B_{\dR}(R^\sharp)\]
    which is $\O_B$-linear and a $\Q_p$-symplectic similitude, and thus also an $F_p$-symplectic similitude with respect to the lifted pairings. Thus $\alpha := (\beta \otimes \id) \circ c^{-1}$ is a trivializing section of $\cF$ over $\Spec B_{\dR}(R^\sharp)$.

    $G(\A_f^p)$-invariance of $\pi_{HT}$ on the system $\{\cX_{\alpha,K^p}^\circ\}_{\alpha,K^p}$ is immediate from the construction. 

    The exact same argument as in the proof of \cite[6.14]{Zha23} also applies to show that the thus constructed Hodge-Tate period maps $\pi_{HT}: \cX_{\alpha,K^p}^\circ \to \Gr_G$ in fact map into the Schubert cell $\Gr =\Gr_{G,\mu}$. 

    The proof of Hecke equivariance at $p$ is also the same as in \cite{Zha23}: Since both $\cX_{\alpha,K^p}^\circ$ and $\Gr$ are qcqs and the latter is also proper, it suffices to check equivariance on rank 1 geometric points $s= \Spa(C,\O_C)$. For $g \in G(\Q_p)$, there is some $N \gg 0$ such that $p^N\Lambda \subset g\Lambda \subset p^{-N}\Lambda$. Denote by $K$ the image of $g\Lambda$ in the quotient
    \[\overline{\beta}: p^{-N}\Lambda/p^N\Lambda \overset{\sim}{\to} \cA_s[p^N],\]
    where $\cA_s$ denotes the generic fiber. Then Hecke action by $g$ sends $(\cA_s,\beta: T_p\cA_s \iso \Lambda)$ to $(\cA'_s := \cA_s/K, T_p\cA'_s \iso g\Lambda)$, which coincides with the action of $G(\Q_p)$ on $\Gr$ that sends a tuple $(\Lambda,W)$ to $(g\Lambda,W)$.
\end{proof}

By construction the Hodge-Tate period map is invariant with respect to changing the tame level, thus we readily obtain a Hodge-Tate period map
\[\pi_{HT}: \cX^\circ \to \Gr\]
at infinite level everywhere.

\begin{corollary}\label{ht_abelian}
    We have a factorization
\[\begin{tikzcd}
	{\cX^\circ} && \Gr \\
	& {\cX_G^\circ}
	\arrow["{\pi_{HT}}", from=1-1, to=1-3]
	\arrow["\Delta"', from=1-1, to=2-2]
	\arrow["{\pi_{HT}}"', dashed, from=2-2, to=1-3]
\end{tikzcd}\]
    In particular, there is a Hodge-Tate period map $\pi_{HT}: \cX_G^\circ \to \Gr$ with the same properties as before.
\end{corollary}
\begin{proof}
    By Lemma \ref{delta_torsor}, the map $\cX^\circ \to \cX_G^\circ$ is a pro-étale torsor with respect to the profinite group $\Delta$, so it suffices to see that the Hodge-Tate period map $\pi_{HT}: \cX^\circ \to \Gr$ is invariant under the action of $\Delta$. This follows from the fact that polarization action by any $\epsilon \in \O_{F,+}^\times$ changes the polarization $\lambda$ to $\epsilon\lambda$, and thus also changes the lifted pairing $\langle\cdot,\cdot\rangle_{M}$ on $M$ by a factor $\epsilon$. This however does not affect the isomorphism sheaf
    \[\underline{\Isom}_G(M \otimes_{W(R^+)} \O_X, \Lambda \otimes_{\Z_p} \O_X),\]
    where we used the same notation as in the proof of Proposition \ref{ht_hybrid}. This shows the first claim.

    The second claim follows immediately from the fact that the Hecke action on $\cX_G^\circ$ is induced from the Hecke action on $\cX^\circ$.
\end{proof}

\section{Stack of G-bundles on the Fargues-Fontaine curve}
In this section we recall the relative Fargues-Fontaine curve and the classifying stack of $G$-torsors on the Fargues-Fontaine curve. In our exposition we will be following \cite{CS17} and \cite{FS21}. Much of the content holds in greater generality, which is why we will assume $G$ to be any connected reductive group over $\Q_p$ for the time being (there is also an obvious generalization to any finite extension of $\Q_p$, see \cite{FS21}).

\subsection{The Fargues-Fontaine curve and vector bundles}
\begin{definition}
    For any affinoid perfectoid space $S = \Spa(R,R^+) \in \Perf$ with pseudo-uniformizer $\varpi \in R^+$, recall that the \textbf{relative Fargues-Fontaine curve over $S$} has three different incarnations:
    \begin{enumerate}[label=(\roman*)]
        \item Let $Y_S := \Spa(W(R^+))\backslash V(p [\varpi])$, then the \textbf{adic Fargues-Fontaine curve} is the quotient $X_S := Y_S / \varphi^\Z$.
        \item The \textbf{diamond Fargues-Fontaine curve} attached to $X_S$ has the formula $X_S^\diamond \iso (S^\diamond \times \Spd\Q_p)/ (\varphi^\Z \times \id)$.
        \item The \textbf{algebraic Fargues-Fontaine} curve is defined as $X_S^{\alg} := \Proj P$, where 
        \[P := \bigoplus_{n \geq 0} H^0(X_S,\O_{X_S}(n))\]
        and $\O_{X_S}(1)$ is the ample line bundle on $X_S$ descending from $\O_{Y_S}$.
    \end{enumerate}
\end{definition}
The above constructions glue and thus yield adic resp. diamond resp. algebraic Fargues-Fontaine curves over any perfectoid space $S \in \Perf$.

Let $S = \Spa(R,R^+) \in \Perf$ and let $B_{\cris,S}^+:= A_{\cris}(R^+)[\frac{1}{p}]$, then we may also write
\[P = \bigoplus_{d \geq 0} (B_{\cris,S}^+)^{\varphi = p^d}.\]
This observation will become useful later, as it will allow us construct vector bundles on the Fargues-Fontaine curve from the rational crystalline Dieudonné module attached to a $p$-divisible group.

From now on we use the following notation: For a perfectoid space $S = \Spa(R,R^+) \in \Perf$, we write 
\[\Y(S) := \Spa(W(R^+),W(R^+)) \backslash \{[\varpi] = 0\}.\]
For any interval $I = [a,b] \subset [0,\infty]$ with $a,b \in \Q \cup \{\infty\}$, we define the open subspace
\[\Y_I(S) := \{|p|^b \leq |[\varpi]| \leq |p|^a\} \subset \Y(S),\]
and similarly for open intervals $(a,b)$. Note that in this notation $Y_S = \Y_{(0,\infty)}(S)$.

By \cite[11.2.1]{SW20} $\Y_S$ is an analytic adic space. It is covered by rational affinoid sousperfectoid subspaces indexed over $n \in \N$ which are of the form 
\[\Spa(R_n,R_n^+) = \{|p| \leq |[\varpi]|^{\frac{1}{p^n}}\},\]
where $R_n^+$ is the $[\varpi]$-adic completion of the ring $W(R^+)[p/[\varpi^{\frac{1}{p^n}}]]$, and $R_n = R_n^+[1/[\varpi]]$. Each element in $R_n$ can be written as
\[\sum_{i \geq 0} [r_i]\biggl(\frac{p}{[\varpi]^{\frac{1}{p^n}}}\biggl)^i\]
with $R \ni r_i \to 0$, which visibily does not depend on $R^+$. Thus the category of vector bundles on $\Spa(R_n, R_n^+)$ does not depend on the choice of $R^+$. By \cite[2.7.7]{KL15} the same holds for $\Y(S)$, and in particular also for $Y_S$ and $X_S$.

The following GAGA-type theorem relates vector bundles on $X_S$ and $X_S^{\alg}$.
\begin{theorem}[\cite{FS21}, II.2.7]
    Pullback along the natural morphism $X_S \to X_S^{\alg}$ of locally ringed spaces induces an equivalence of categories between the categories of vector bundles on $X_S$ and $X_S^{\alg}$.
\end{theorem}

Importantly, Cartier divisors on $X_S$ classify Frobenius orbits of untilts of $S$ as follows. Assume $S^\sharp / \Q_p$ is an untilt of $S$, which is locally of the form $\Spa(R^\sharp, R^{\sharp +})$. Then each kernel of the surjection induced from Fontaine's theta map
\[W(R^+) \iso W(R^{\sharp\flat +}) \overset{\theta}{\twoheadrightarrow} R^{\sharp +}\]
is a principal ideal generated by an element of the form $p - a[\varpi]$ for some $a \in W(R^+)$. The induced maps $\Spa(R^\sharp, R^{\sharp +}) \to Y_{\Spa(R,R^+)}$ glue and define a closed Cartier divisor $S^\sharp \hookrightarrow Y_S$, which induces a closed Cartier divisor $S^\sharp \hookrightarrow X_S$. This is cut out by a global section of $\O_{X_S}(1)$, which by GAGA corresponds to a global section of the line bundle $\O(1)$ on $X_S^{\alg}$, which in turn cuts out a closed Cartier divisor $S^{\sharp,\alg} \hookrightarrow X_S^{\alg}$.

If $S = \Spa(R,R^+)$ is affinoid, then the algebraic Fargues-Fontaine curve $X_S^{\alg}$ can be covered by any two principal opens $D(f_1)$ and $D(f_2)$ with $f_1, f_2 \in H^0(X_S^{\alg},\O(1))$ linearly independent. In particular, if the untilt $S^\sharp$ is cut out by $\xi \in H^0(X_S^{\alg},\O(1))$, then choose any linearly independent $t \in H^0(X_S^{\alg},\O(1))$. The Cartier divisor $S^{\sharp, \alg} \hookrightarrow X_S^{\alg}$ is defined by
\[(P[1/t])_0 \twoheadrightarrow R^\sharp\]
with kernel $(\xi)$. In the meantime, as shown in \cite{FF18}, the $\xi$-adic completion of $P[1/t]_0$ is precisely the de Rham period ring $B_{\dR}^+(R^\sharp)$.

This appearance of the de Rham period ring leads to an interpretation of the $B_{\dR}^+$-affine Grassmannian $\Gr_{\GL_n}$ in terms of modifications of the trivial rank $n$ vector bundle on $X_S$. As a last ingredient for this, we will need the Beauville-Laszlo lemma:

\begin{lemma}[Beauville-Laszlo]
    Let $R$ be a commutative ring, $f \in R$ a non-zerodivisor and $\hat{R} := \invlim R/f^n$ the $f$-adic completion of $R$. Then the category of $R$-modules on which $f$ is a non-zerodivisor is equivalent to the category of triples 
    \[(M_1,M_2, \alpha: M_1[1/f] \overset{\sim}{\to} M_2 \otimes_R \hat{R}),\]
    where $M_1$ is an $\hat{R}$-module on which $f$ is a non-zerodivisor, and $M_2$ is an $R[1/f]$-module. 
\end{lemma}

For any $S = \Spa(R,R^+) \in \Perf_{\Q_p}$, an $S$-point of $\Gr_{\GL_n}$ corresponds to an untilt $S^\sharp = \Spa(R^\sharp, R^{\sharp +})$, a rank $n$ vector bundle $\cF$ on $\Spec(B_{\dR}^+(R^\sharp))$ and a trivialization of $\cF$ over $B_{\dR}(R^\sharp)$. By Beauville-Laszlo, we can glue the trivial rank $n$ vector bundle on $X_S^{\alg} \backslash S^{\sharp,\alg}$ with the vector bundle $\cF$ over $B_{\dR}^+(R^\sharp)$ to obtain a new rank $n$ vector bundle on $X_S^{\alg}$, which by GAGA corresponds to a rank $n$ vector bundle on $X_S$. This is a modification of the trivial rank $n$ vector bundle on $X_S$ at the point (corresponding to) $S^\sharp$.

\subsection{The stack of $G$-bundles}
In this subsection we denote by $X$ either a scheme over $\Q_p$ or a sousperfectoid space over $\Spa\Q_p$. By $\Rep(G)$ (resp. $\Bun(X)$) we denote the exact symmetric monoidal category of finite-dimensional algebraic representations of $G$ over $\Q_p$ (resp. of vector bundles on $X$).

By a $G$-bundle on $X$ we mean an étale sheaf on $X$ together with an action by $G$ that is $G$-equivariantly isomorphic to the trivial $G$-torsor étale locally on $X$. From the Tannakian point of view (see \cite[19.5.1, 19.5.2]{SW20}) this is equivalent to giving an exact tensor functor
\[\Rep(G) \to \Bun(X)\]

By \cite[Proof of 11.2.1]{SW20}, for any perfectoid $S \in \Perf_{\Q_p}$ its relative Fargues-Fontaine curve $X_S$ is sousperfectoid. Thus we can talk about $G$-torsors on $X_S$ in terms of exact tensor functors. Since we also have the exact tensor equivalence $\Bun(X_S) \iso \Bun(X_S^{\alg})$, the categories of $G$-torsors on $X_S$ and $X_S^{\alg}$ are equivalent.

\begin{prop}[\cite{FS21}, II.2.1, III.1.2, III.1.3]
    The pre-stack on $\Perf_{\Q_p}$ that maps a perfectoid space $S \in \Perf_{\Q_p}$ to the groupoid of $G$-torsors on $X_S$ is a small v-stack.
\end{prop}
This v-stack, denoted by $\Bun_G$, is called the \textbf{stack of $G$-bundles on the Fargues-Fontaine curve}.

The interpretation of $\Gr_{\GL_n}$ as classifying modifications of the trivial rank $n$ vector bundle of on the Fargues-Fontaine curve can be generalized to $G$-torsors, as we will explain next.

Let $S = \Spa(R,R^+) \in \Perf_{\Q_p}$ be a perfectoid space with untilt $S^\sharp$ over $\Q_p$, which we view as a Cartier divisor on $X_S$. By Beauville-Laszlo, there is an exact symmetric monoidal equivalence
\[\Bun(X_S) \iso \Bun(X_S^{\alg}) \iso \Bun(X_S^{\alg} \backslash S^{\sharp,\alg}) \times_{\Bun(\Spec(B_{\dR}(R^\sharp)))} \Bun(\Spf(B_{\dR}^+(R^\sharp)))\]
Hence there is an equivalence between the category of exact tensor functors from $\Rep(G)$ to the left-hand side and to the right-hand side. 

Now assume we are given an $S$-valued point  $(\cF,\alpha)$ of $\Gr_G$, where $\cF$ is a $G$-torsor over $B_{\dR}^+(R^\sharp)$ and $\alpha$ is a trivialization of $\cF$ over $B_{\dR}(R^\sharp)$. By the above exact symmetric monoidal equivalence, we can glue $\cF$ and the trivial $G$-torsor along the trivialization $\alpha$ to obtain a new $G$-torsor on $X_S^{\alg}$, or equivalently on $X_S$.

This construction yields the \textbf{Beauville-Laszlo uniformization morphism}
\[BL: \Gr_G \to \Bun_G.\]

\begin{prop}[\cite{FS21}, III.3.1]
    $BL$ is surjective as a map of pro-étale stacks.
\end{prop}

\subsection{Explicit computations}\label{tors_vbs}
We now return to our initial PEL setup and consider the v-stacks $\Bun_{G^*}$ and $\Bun_G$ (we are abusing notation as before and speak of $G$-torsors when we should speak of $G_{\Q_p}$-torsors, similarly for $G^*$). We can use the interpretation of $G^*$ and $G$ as groups of certain symplectic similitudes, together with the the Tannakian formalism, to describe torsors on the Fargues-Fontaine curve with respect to $G^*$ and $G$ in terms of vector bundles.

Namely, given a $G^*$-torsor on $X_S$, we can interpret it as a vector bundle $\ccE$ on $X_S$ with an $\O_B$-module structure and a symplectic pairing $(\cdot,\cdot)_\ccE$ that is $\Q_p$-bilinear, $\Q_p$-valued and that satisfies
\[(ax,y)_\ccE = (x,ay)_\ccE\]
for any local sections $x,y \in \ccE$ and $a \in F_p$, i.e. by Lemma \ref{forms_trace} it is liftable along the trace $\Tr_{F_p/\Q_p}$ to an $F_p$-bilinear, $F_p$-valued symplectic pairing $\langle\cdot,\cdot\rangle_\ccE$. Morphisms of $G^*$-torsors are then morphisms of vector bundles that are compatible with the $\O_B$-module structure and that are $\Q_p$-symplectic similitudes. 

Similarly, given a $G$-torsor on $X_S$, we can interpret it as a vector bundle $\ccE$ on $X_S$ with an $\O_B$-module structure and a symplectic pairing $\langle\cdot,\cdot\rangle_\ccE$ that is $F_p$-bilinear and $F_p$-valued. Morphisms between $G$-torsors are then morphisms between vector bundles that are compatible with the $\O_B$-module structure and that are $F_p$-symplectic similitudes. 

In both cases, we obtain the vector bundle from the associated torsor by pushforwarding it along the canonical faithful representation $G^* \hookrightarrow G \hookrightarrow \GL_\Q(V)$. To get the torsor back from the vector bundle, we can consider isomorphism sheaves $\underline{\Isom}_{G^*}(-,\Lambda \otimes_{\Z_p} \O_{X_S})$ resp. $\underline{\Isom}_G(-,\Lambda \otimes_{\Z_p} \O_{X_S})$. 

We can pushforward $G^*$-torsors on $X_S$ along the embedding $G^* \hookrightarrow G$ to obtain $G$-torsors. In terms of vector bundles, this corresponds precisely to lifting the pairing $(\cdot,\cdot)_\ccE$ along the trace $\Tr_{F_p/\Q_p}$ to the pairing $\langle\cdot,\cdot\rangle_\ccE$. Concretely, given a $G^*$-torsor $\underline{\Isom}_{G^*}(\ccE,\Lambda \otimes \O_{X_S})$, it is easy to see that 
\[\underline{\Isom}_G(\ccE,\Lambda \otimes_{\Z_p} \O_{X_S}) = \underline{\Isom}_{G^*}(\ccE,\Lambda \otimes_{\Z_p} \O_{X_S}) \times^{G^*} G.\]

One word of warning: keeping Lemma \ref{forms_trace} in mind, it might seem natural to also go in the other direction by composing $\langle\cdot,\cdot\rangle_{\ccE}$ with the trace $\Tr_{F_p/\Q_p}$ to get a $G^*$-torsor back. However, this will of course not define a map $\Bun_G \to \Bun_{G^*}$, because $F_p$-symplectic similitudes with respect to the lifted pairings do not induce $\Q_p$-symplectic similitudes after composing with $\Tr_{F_p/\Q_p}$. 

\section{Construction of the Igusa stack}\label{igs_chapter}
After having discussed all the necessary background, we are now in a position to prove the fiber product conjecture on the good reduction locus of our abelian-type Shimura varieties.

\subsection{Hybrid Igusa stack}
In this subsection we define the hybrid Igusa stack for varying parameter $\alpha \in \A_{F,f}^{p,\times}$, and show that it yields a similar fiber product diagram as the one we want for our abelian-type Shimura varieties.

In the definition of hybrid Igusa stacks, we have drawn inspiration from the definition of Hodge-type Igusa stacks in \cite{DHKZ24} and also from the description of Igusa varieties in the setting of Hilbert modular varieties, see \cite[4.2.2]{CT23}.

We have to make the following auxiliary choice: for any prime $\P$ of $\O_F$ lying above $p$, we fix an element $x_\P \in \O_{F,+}$ such that $v_\P(x_\P) = 1$ and such that for any $\P' \neq \P$ lying above $p$, $v_{\P'}(x_\P) = 0$, i.e. $x_\P$ is a totally positive uniformizer for $\P$ that is a unit with respect to any other prime $\P'$ lying above $\P$. 

We will not include this choice in our notation, although the construction of the hybrid Igusa stack seems to depend on it. However, the abelian-type Igusa stack that we will construct out of the hybrid one will not depend on this choice (up to unique isomorphims), because of the main results of \cite{Kim25}.

\begin{definition}\label{hybrid_igs}
    Fix $\alpha \in \A_{F,f}^{p,\times}$. We define a presheaf of groupoids $\Igs_{\alpha}^{\pre}: \Perf \to \Grpd$ by defining the groupoid $\Igs_{\alpha}^{\pre}(S)$ for any affinoid perfectoid $S = \Spa(R,R^+) \in \Perf$ as follows.

    Its objects are pairs $(S^\sharp,x)$, where $S^\sharp = \Spa(R^\sharp,R^{\sharp+})$ is an untilt of $S$ over $\O_E$, and $x: \Spf R^{\sharp+} \to \X_\alpha$ is a map of formal schemes over $\O_E$, which corresponds to an abelian scheme $\fA_x$ with $G$-structure and infinite away-from-$p$ level structure over $\Spf R^{\sharp+}$.

    A morphism $f: (S^{\sharp_1},x) \to (S^{\sharp_2},y)$ is a quasi-isogeny $\fA_x \times_{R^{\sharp_1+}} R^+/\varpi \to \fA_y \times_{R^{\sharp_2+}} R^+/\varpi$ that respects the infinite level structure away from $p$, preserves the $\O_B$-structure and that preserves the polarizations up to a scalar in the $\Spec R^+/\varpi$-points of the locally constant sheaf $\underline{\Z_{(p)}^\times \prod_{\P|p} x_\P^\Z}$.
\end{definition}
In order to see that our definition of morphisms is well-defined, we refer to Remark \ref{multiple_untilts}. Namely, we may choose a small enough pseudo-uniformizer $\varpi$ such that $R^{\sharp_i+}$ surjects onto $R^+/\varpi$ for $i=1,2$. By Serre-Tate lifting,  this does not depend on the choice of pseudo-uniformizer $\varpi$.

As before, we will drop the parameter $\alpha$ from our notation if it is understood from context or does not matter. 

The following result is the exact same as in the Hodge-type setting.
\begin{lemma}[\cite{DHKZ24}, 5.2.6]\label{igs_sheaf}
    $\Igs^{\pre}$ is a $0$-truncated presheaf of groupoids, i.e. objects in the groupoid $\Igs^{\pre}(S)$ do not have non-trivial automorphisms for any $S = \Spa(R,R^+) \in \Perf$
\end{lemma}
\begin{proof}
    For convenience, we recall the proof. It suffices to prove that for any $\Z_p$-algebra $R$ and any abelian scheme $\pi: A \to \Spec R$ the natural map
    \[\End_R(A) \otimes_\Z \Q \to \End_R(V^pA)\]
    is injective, where $V^pA$ denotes the prime-to-$p$ adelic Tate module of $A$. Actually, it suffices to prove that the natural map
    \[\End_R(A) \to \End_R(T^pA)\]
    is injective, where $T^pA$ denotes the prime-to-$p$ Tate module of $A$ over $R$. Let $f\in \End_R(A)$ be in the kernel of this map.

    Write $R = \dirlim R_i$ as the filtered colimit of finitely generated $\Z$-algebras $\{R_i\}_{i \in I}$ over some indexing set $I$. By spreading out, there is some index $i \in I$, an abelian scheme $\pi_i: A_i \to \Spec R_i$ and an endomorphism $f_i \in \End_{R_i}(A_i)$ such that base changing $(A_i,f_i)$ to $R$ gives us back $(A,R)$. Since $|\Spec R | = \invlim |\Spec R_i|$, we may also assume without loss of generality that the image of $\Spec R \to \Spec R_i$ meets all of the finitely many connected compoentns of $\Spec R_i$. We need to see that $f_i = 0$.

    The result is well-known over algebraically closed fields, thus for every geometric point $s: \Spec k \to \Spec R$, we find that $f_s = 0$. By our assumption on the map $\Spec R \to \Spec R_i$, we find for any connected component of $\Spec R_i$ a geometric point $s: \Spec k \to \Spec R_i$ such that $f_{i,s} = 0$. By rigidity of abelian schemes, this already implies $f_i = 0$.
\end{proof}

\begin{definition}
    The hybrid Igusa stack $\Igs$ at infinite level is defined to be the v-sheafification of the presheaf that maps $S$ to the set of isomorphism classes of the groupoid $\Igs^{\pre}(S)$.
\end{definition}

\begin{lemma}
    There is a map of v-sheaves $q: \cX^\circ \to \Igs$.
\end{lemma}
\begin{proof}
    It suffices to construct a map of presheaves $\cX^{\circ,\pre} \to \Igs^{\pre}$ and then sheafify. Recall from Corollary \ref{grlocus_moduli_cor} that $\Spa(R,R^+)$-points of $\cX^{\circ,\pre}$ correspond to triples $(S^{\sharp},y,\beta)$, where $S^{\sharp}$ is an untilt of $S$, $y: \Spf R^{\sharp+} \to \X$ is a map of formal schemes, and $\beta$ is a trivialization of the Tate module. Forgetting about $\beta$ yields an obvious map to $\Igs^{\pre}$.
\end{proof}

The map $\overline{\pi}_{HT}: \Igs \to \Bun_G$ can be constructed analogously as in \cite{Zha23}.

\begin{prop}
    There is a map $\overline{\pi}_{HT}: \Igs \to \Bun_G$.
\end{prop}
\begin{proof}
    Since $\Bun_G$ is a v-stack, it suffices to define a map of prestacks $\Igs^{\pre} \to \Bun_G$ and then sheafify. Let $S = \Spa(R,R^+)$, then we wish to define a functorial map of groupoids $\Igs^{\pre}(S) \to \Bun_G(S)$.

    Given an untilt $S^\sharp = \Spa(R^{\sharp}, R^{\sharp+})$ and a formal abelian scheme $\fA$ over $\Spf R^{\sharp+}$ with $G$-structure and infinite away-from-$p$ level structure, choose a sufficiently small pseudo-uniformizer $\varpi \in R^+$ such that we have a quotient map $R^{\sharp+} \twoheadrightarrow R^+/\varpi$. Consider the reduction $A := \fA \times_{R^{\sharp+}} R^+/\varpi$, which is still equipped with a $G$-structure, as is its $p$-divisible group $A[p^\infty]$. We may attach to it the rational crystalline Dieudonné module $M[p^{-1}] := M_{crys}(A[p^\infty])[p^{-1}]$, which is a finite locally free $B_{\cris}^+(R^+/\varpi)$-module that is well-defined up to unique isomorphism.

    By fully faithfulness of the crystalline Dieudonné module functor, the module $M[p^{-1}]$ comes equipped with an $\O_B$-module structure and with a symplectic $F_p$-bilinear pairing. Using GAGA for the Fargues-Fontaine curve $X_S$, the graded module
    \[\bigoplus_{d \geq 0} M[p^{-1}]^{\varphi = p^{d+1}}\]
    defines a vector bundle $\ccE(A)$ on $X_S$, which also comes equipped with an $\O_B$-module strcutreu and with a symplectic $F_p$-valued pairing. Consider the isomorphism sheaf on the Fargues-Fontaine curve
    \[\underline{\Isom}_G(\ccE(A),\Lambda \otimes_{\Z_P} \O_{X_S})\]
    of $\O_B$-linear trivializations of $\ccE(A)$ that are $F_p$-symplectic similitudes with respect to the symplectic pairing on $\ccE(A)$ and $(\cdot,\cdot)_F$ on $\Lambda$. We claim that this sheaf is a $G$-torsor and thus yields an object in $\Bun_G(S)$

    To prove the claim, note that by \cite[8.6]{Zha23} the vector bundle $\ccE(A)$ is sitting in the short exact sequence
    \[0 \to T_p\cA \otimes \O_{X_S} \to \ccE(A) \to i_* \Lie(\cA) \to 0\]
    where $\cA$ denotes the adic generic fiber of $\fA$, $T_p\cA$ denotes its Tate module, and $i: S^{\sharp} \to X_S$ is the closed immersion of the Cartier divisor $S^{\sharp}$ of $X_S$. $T_p\cA \otimes \O_{X_S}$ is étale locally isomorphic to $\Lambda \otimes \O_{X_S}$, and the completion of $\ccE(A)$ at $S^{\sharp}$ is étale locally isomorphic to the completion of $\Lambda \otimes \O_{X_S}$ at $S^{\sharp}$, bot as $\O_B \otimes \O_{X_S}$ modules with symplectic pairings. This shows that the above isomorphism sheaf is trivialized étale locally.

    Lastly, any morphism in $\Igs^{\pre}(S)$ induced a quasi-isogeny $A_1[p^\infty] \to A_2[p^\infty]$ and thus an isomorphism $M_1[p^{-1}] \iso M_2[p^{-1}]$ of rational crystalline Dieudonné modules with $\O_B$-structure and symplectic pairing. This yields an isomorphism between the associated $G$-torsors, and everything is functorial in $S$.
\end{proof}

We have the following notion of tame Hecke action by $G(\A_f^p)$ on the system of hybrid Igusa stacks $\{\Igs_{\alpha}\}_{\alpha \in \A_{F,f}^{p,\times}}$. Namely, for $g \in G(\A_f^p)$ we have the isomorphism
\[[g]: \Igs_\alpha^{\pre} \to \Igs_{\nu(g)\alpha}^{\pre},\]
which is given by acting on the infinite away-from-$p$ level structure in the moduli description. We may sheafify to obtain an isomorphism
\[[g]: \Igs_{\alpha} \overset{\sim}{\to} \Igs_{\nu(g)\alpha}.\]
We can extend this action by letting $G(\Q_p)$ act trivially on $\Igs_\alpha$, in order to get a full Hecke action by $G(\A_f)$ on $\{\Igs_\alpha\}_{\alpha \in \A_{F,f}^{p,\times}}$.

By construction, the maps $\overline{\pi}_{HT}: \Igs_\alpha \to \Bun_G$ are Hecke invariant, and the following diagram is commutative on the nose;
\[\begin{tikzcd}
	{\cX_{\alpha}^\circ} & {\cX_{\nu(g)\alpha}^\circ} \\
	{\Igs_{\alpha}} & {\Igs_{\nu(g)\alpha}}
	\arrow["{[g]}", from=1-1, to=1-2]
	\arrow["q"', from=1-1, to=2-1]
	\arrow["q", from=1-2, to=2-2]
	\arrow["{[g]}"', from=2-1, to=2-2]
\end{tikzcd}\]

The proof of the desired fiber product for the hybrid space will closely follow the argument in \cite[section 8.2]{Zha23}. For this, we need the following two preparatory lemmas. 

\begin{lemma}[\cite{Zha23}, 8.10]\label{extension_lemma}
    Let $V$ be a valuation ring with fraction field $K$, let $A$ be an abelian scheme over $V$ with generic fiber $A_K$. Assume $j: G_K \hookrightarrow A_K$ is a finite sub-group scheme. Then there exists a finite sub-group scheme $G$ of $A$, flat over $V$, such that its generic fiber agrees with $G_K$.
\end{lemma}
\begin{proof}
    We recall the proof. Since the map $j$ is quasi-compact, we can let $Z \subset A$ be the closed subscheme cut out by the quasi-coherent ideal $\cI := \ker(\O_A \to j_* \O_{G_K})$. Denote by $\O_Z$ the quotient $\O_A/\cI$, considered as an $\O_A$-module. Then we have
    \[\O_A \twoheadrightarrow \O_Z \hookrightarrow j_*\O_{G_K}.\]
    Since $j_*\O_{G_K}$ is torsion-free, $\O_Z$ is also torsion-free over the valuation ring $V$ and thus flat. 

    It now suffices to equip $Z$ with a group scheme structure that is compatible with the one on $A$. Denote by $m: A \times_{\Spec V}A \to A$ multiplication on $A$. To define multiplication on $Z$, note that $\O_Z(Z)$ is finitely generated over the valuation ring $V$, thus it is a finite projective $V$-module, which implies that the surjection $\O_A \twoheadrightarrow \O_Z$ splits as $V$-modules. We can now define the multiplication map $m_Z^*$ as the composition
    \[\O_Z \to \O_A \overset{m_A^*
    }{\to} m_*\O_{A \times A} \to m_*\O_{Z \times Z}.\]
    A priori this is only a morphism of $V$-modules, but in fact it is a homomorphism of $V$-algebras. This is because after postcomposition with the injection $\O_{Z \times Z} \hookrightarrow (j \times j)_*\O_{G_K \times G_K}$, the map $m_Z^*$ agrees with the map
    \[\O_Z \hookrightarrow j_*\O_{G_K} \overset{m^*}{\to} (j \times j)_*\O_{G_K \times G_K},\]
    which is a $V$-algebra homomorphism. This defines multiplication on $Z$, and it is also clearly the restriction of multiplication on $A$ to $Z$.

    Since the inverse map on $A$ preserves $G_K$, it also preserves its schematic image $Z$, and restricting it to $Z$ yields the inverse map on $Z$.

    Finally, the identity section $\Spec K \to G_K$ extends to a section $\Spec V \to Z$ by properness of $Z$ over $V$, which by uniqueness must be the same as the identity section on $A$.
\end{proof}

The second lemma is an adaptation of \cite[Lemma 8.11]{Zha23}.
\begin{lemma}\label{8.11 adaptation}
    Let $S$ be an $\O_E$-scheme and $\cA = (A,\iota_A,\lambda_A,\eta) \in \M(S)$. Let $\cH = (\cH/S,\iota_H,\lambda_H)$ be a $p$-divisible group with $G$-structure over $S$. Assume that
    \[\rho: \cA[p^\infty] \to \cH\]
    is an $\O_B$-linear isogeny preserving the polarizations up to a scalar in $\underline{F_p^\times}(S)$. Then $A' := A/\ker\rho$ can be uniquely promoted to an object $(A',\iota_A',\lambda_A',\eta') \in \M(S)$ (i.e. up to isomorphism in the moduli description of $\M(S)$), such that the induced map
    \[\rho': A'[p^\infty] \to \cH\]
    is an $\O_p$-isomorphism of $p$-divisible groups with $G$-structure, and such that the quotient map $\pi: A \to A'$ is $\O_B$-linear, preserves the infinite away-from-$p$ level structures and preserves the polarizations up to a scalar in $\underline{\Z_{(p)}^\times \prod_{\P|p} x_\P^{\Z}}(S)$.
\end{lemma}
\begin{proof}
    Since $A'$ is a modification of $A$ at $p$, $A'$ clearly inherits an infinite away-from-$p$ level structure from $A$. Namely, the condition on the quotient map $\pi: A \to A'$ forces $\eta'$ to be $\eta$, and also forces us to define $\iota_A' := \pi\iota_A\pi^{-1}$. Note that $\iota_A'$ thus defined is indeed valued in prime-to-$p$ quasi-isogenies, because on the $p$-divisible groups $\iota_A'(b)[p^\infty] = \iota_H(b)$ is an isomorphism. 

    As for the polarization, by the condition on $\pi$ we are forced to define $\lambda_A' = (\pi^\vee)^{-1} (d\lambda_A)\pi$ for some $d \in \underline{\Z_{(p)}^\times \prod_{\P|p} x_\P^\Z}(S)$. Assume $c \in \underline{F}_p^\times(S)$ is a scalar such that 
    \[c \lambda_A[p^\infty] = \rho^\vee \lambda_H \rho,\]
    then $\rho'$ preserves the polarization up to $cd^{-1} \in \underline{F}_p^\times(S)$. Since we want $\rho'$ to be an $\O_p$-isomorphism of $p$-divisible groups with $G$-structure, we need to require $cd^{-1} \in \underline{\O}_p^\times(S)$, i.e. for all primes $\P$ lying above $p$ we need to require $v_\P(c) = v_\P(d)$. This uniquely determines $d \in \underline{\Z_{(p)}^\times \prod_{\P|p} x_\P^{v_\P(c)}}(S)$, up to a scalar in $\underline{\Z_{(p)}^\times}(S)$, as desired. 
\end{proof}

\begin{remark}
    If in the above Lemma $\rho$ is only a quasi-isogeny, then we may assume that for some large $N$, $p^N\rho$ is an actual isogeny. Then
    \[\cA' := (A/\ker(p^N\rho), \pi \iota_A \pi^{-1}, \pi^\vee d\lambda_A \pi, \overline{\eta}')\]
    is a canonical object of $\M_{K}(S)$.
\end{remark}

Let $F := \Igs \times_{\Bun_G} \Gr$. Since v-stackification commutes with fiber products, it will be convenient to instead consider the presheaf of groupoids 
\[F^{\pre} := \Igs^{\pre} \times_{\Bun_G} \Gr.\]
By Lemma \ref{igs_sheaf}, $F^{\pre}$ is 0-truncated.

\begin{lemma}\label{diagram_is_commutative}
    The diagram
\[\begin{tikzcd}
	{\cX^{\circ,\pre}} & \Gr \\
	{\Igs^{\pre}} & {\Bun_G}
	\arrow["{\pi_{HT}}", from=1-1, to=1-2]
	\arrow["q"', from=1-1, to=2-1]
	\arrow["BL", from=1-2, to=2-2]
	\arrow["{\overline{\pi}_{HT}}"', from=2-1, to=2-2]
\end{tikzcd}\]
is 2-commutative.
\end{lemma}
\begin{proof}
    Given a point $(S^\sharp,\fA,\beta)$ of $\cX^{\circ,\pre}$, let $A = \fA \times_{R^{\sharp+}} R^+/\varpi$ be the reduction of $\fA$ modulo $\varpi$. Let $\cG = \fA[p^\infty]$. Then $\cG \times_{R^{\sharp+}} R^+/\varpi = A[p^\infty]$.

    Tracing along the upper-right path of the diagram gives the $G$-bundle associated to
    \[M_{\Prism}(\cG) \otimes_{W(R^+)} A_{\cris}(R^+/\varpi)\]
    and tracing along the lower-left path of the diagram gives the $G$-bundle associated to
    \[M_{\cris}(\cG \times_{R^{\sharp+}} R^+/\varpi)\]
    Thus 2-commutativity is established by the comparison isomorphism between prismatic (or more specifically Breuil-Kisin-Fargues modules in this case) and crystalline Dieudonné modules, see \cite[17.5.2]{SW20}.
\end{proof}

Thus we obtain a natural map $\cX^{\circ,\pre} \to \Igs^{\pre} \times_{\Bun_G} \Gr$, which we want to show is an equivalence after v-sheafification. It suffices to check that that the map is bijective after evaluating at strictly totally disconnected perfectoid spaces, because the latter form a basis for the v-topology on $\Perf$.

Assume $S = \Spa(R,R^+)$ is a product of points, i.e. $R^+ = \prod_{i \in I} C_i^+$, where each $C_i^+$ is a valuation subring of a complete non-archimedean algebraically closed field $C_i$ (such that $C_i^+ \subset \O_{C_i}$), with pseudo-uniformizers $\varpi_i \in C_i^+$. Let $\varpi = (\varpi_i)$, then $R = R^+[1/\varpi]$. By $k_i$ we will denote the residue field of $C_i$, and write $\overline{C_i^+}$ for the image of $C_i^+$ in $k$. Whenever we put a subscript to a morphism on $S$, we will mean its restriction along $\Spa(C_i,C_i^+) \to S$.

\begin{prop}
    For any $S$ as above, the natural map
    \[\cX^{\circ,\pre}(S) \to F^{\pre}(S)\]
    is a bijection that is natural in $S$.
\end{prop}
\begin{proof}
    As a warm-up, we start with the case of a single geometric point $S = \Spa(C,C^+)$, where $C^+ \subset \O_C$ is a bounded valuation subring. Let $k$ be the residue field of $C$, and let $\overline{C^+}$ denote the image of $C^+$ in $k$. 

    Let us unpack what a point in $(\Igs^{\pre} \times_{\Bun_G} \Gr)(C,C^+)$ looks like. A $\Spa(C,C^+)$-point of $\Gr$ is equivalent to the data of an untilt $(C^{\sharp_2},C^{\sharp_2+})$ of $(C,C^+)$ over $E$ and a $p$-divisible group $\cH$ over $\O_{C^{\sharp_2}}$ with $G$-structure that comes equipped with a trivialization $\alpha$ of its Tate module. Secondly, a $\Spa(C,C^+)$-point of $\Igs^{\pre}$ is given by another untilt $(C^{\sharp_1}, C^{\sharp_1+})$ of $(C,C^+)$ over $\O_E$ together with a map $\Spf C^{\sharp_1+} \to \X$. The latter data is equivalent to giving a formal abelian scheme $\fA$ over $C^{\sharp_1+}$ with $G$-structure and infinite prime-to-$p$ level structure. 

    Fix a sufficiently small pseudo-uniformizer $\varpi \in C^+$ such that we get maps $C^{\sharp_1+} \twoheadrightarrow C^+/\varpi$ and $C^{\sharp_2+} \twoheadrightarrow C^+/\varpi$, and let $\cA = \fA \times_{C^{\sharp_1+}} C^+/\varpi$, which is an abelian scheme with $G$-structure and infinite prime-to-$p$ level over $C^+/\varpi$. Then compatibility over $\Bun_G$ is given by the data of an $F_p$-quasi-isogeny
    \[\rho: \cA[p^\infty] \times_{C^+/\varpi} \O_C/\varpi \to \cH \times_{\O_{C^{\sharp_2}}} \O_C/\varpi,\]
    by fully faithfulness of $\ccE(\cdot)$ on $p$-divisible groups over $\O_C/\varpi$. We will now fix such a data. 

    From here on the argument works verbatim as in \cite{Zha23}, but for the convenience of the reader, we repeat it with the appropriate adjustments (in particular with the adjustments of Lemma \ref{8.11 adaptation}).
    
    We may assume that $p^n\rho$ is an actual isogeny for some integer $n$ (but still preserving the polarizations only up to a scalar in $F_p^\times$). Apply Lemma \ref{extension_lemma} to the valuation ring $\overline{C^+} \subset k$ with fraction field $k$, the abelian scheme $\cB := \cA \times_{C^+/\varpi} \overline{C^+}$ and the finite group scheme $\ker(p^n\rho_k)$, where $\rho_k$ is the base change of $\rho$ to $k$. It shows that the closure of $\ker(p^n\rho_k)$ in $\cB$ is a finite flat group scheme. We may take the quotient by it to obtain a new abelian scheme $\cB'$, which by our Lemma \ref{8.11 adaptation} can be uniquely (in $\M(\overline{C^+})$) endowed with a $G$-structure and infinite away-from-$p$ level. Note that we are not exactly in the setting of Lemma \ref{8.11 adaptation}, as $\rho_k$ is only defined over $k$ and not over $\overline{C^+}$, but since $k = \Frac(\overline{C^+})$ this suffices to fix the constant in the construction of the polarization. Also, Lemma \ref{8.11 adaptation} yields an $\O_p$-isomorphism $\cB'[p^\infty]_k \iso \cH_k$, and the quotient map $\cB \to \cB'$ preserves the polarizations up to a factor in $\underline{\Z_{(p)}^\times \prod_{\P|p} x_\P^\Z}$.

    Next, we can use the Milnor square
    \[\begin{tikzcd}
	   {C^{\sharp_2 +}} & {\overline{C^+}} \\
	   {\O_{C^{\sharp_2}}} & k
	   \arrow[from=1-1, to=1-2]
	   \arrow[from=1-1, to=2-1]
	   \arrow[from=1-2, to=2-2]
	   \arrow[two heads, from=2-1, to=2-2]
    \end{tikzcd}\]
    and Proposition \ref{BT_milnor} to glue $\cB'[p^\infty]$ and $\cH$ to get a $p$-divisible group $\cG$ with $G$-structure over $C^{\sharp_2+}$, with a trivialization $\psi: T_p\cG \iso T_p\cH \iso \underline{\Lambda}$. Also, using the Milnor square

    \[\begin{tikzcd}
	   {C^+/\varpi \cdot \O_C} & {\overline{C^+}} \\
	   {\O_C/\varpi} & k
	   \arrow[from=1-1, to=1-2]
	   \arrow[from=1-1, to=2-1]
	   \arrow[from=1-2, to=2-2]
	   \arrow[two heads, from=2-1, to=2-2]
    \end{tikzcd}\]
    and the fully faithfulness of Proposition \ref{BT_milnor}, applied to the isogenies
    \[p^n\rho: \cA[p^\infty] \times_{C^+/\varpi} \O_C/\varpi \to \cH \times_{\O_{C^{\sharp_2}}} \O_C/\varpi\] 
    and 
    \[\cB[p^\infty] = \cA[p^\infty] \times_{C^+/\varpi} \overline{C^+} \to \cB'[p^\infty],\] 
    we obtain an isogeny
    \[\cA[p^\infty] \times_{C^+/\varpi} (C^+/\varpi \cdot \O_C) \to \cG \times_{C^{\sharp_2 +}} (C^+/\varpi \cdot \O_C).\]
    preserving the polarization up to a factor in $F_p^\times$. By Serre-Tate lifting (see Theorem \ref{st_2}), this lifts to an isogeny
    \[\cA[p^\infty] \to \cG \times_{C^{\sharp_2+}} C^+/\varpi,\]
    again preserving the polarizations up to a factor in $F_p^\times$. Multiplying this isogeny with $p^{-n}$ yields an $F_p$-quasi-isogeny that is not dependent on choice of $n$.

    At this point, we revert back to the general case where $S = \Spa(R,R^+)$ is a product of points $s_i = \Spa(C_i,C_i^+)$ (over an arbitrary index set $i \in I$). Given an $S$-point of $\Igs \times_{\Bun_G} \Gr$
    \[((R^{\sharp_1},R^{\sharp_1+}),\fA,(R^{\sharp_2}, R^{\sharp_2+}), y, \phi: \ccE(\cA) \iso \ccE(y)),\]
    fix a uniformizer $\varpi \in R^+$ sufficiently small as before, and let $\cA = \fA \times_{R^{\sharp_1+}} R^+/\varpi$. For each $i \in I$ we can pull this $S$-point back to an $s_i$-point and proceed as before in the case of a single geometric point. 

    This way, we obtain for each $i \in I$ a $p$-divisible group $\cG_i$ with $G$-structure over $C_i^{\sharp_2+}$, a trivialization $\psi_i: T_p\cG_i \iso \underline{\Lambda}$ and an $F_p$-quasi-isogeny
    \[\rho_i: \cA[p^\infty] \times_{R^+/\varpi} C_i^+/\varpi_i \longrightarrow \cG_i \times_{C_i^{\sharp_2 +}} C_i^+/\varpi_i,\]
    such that for all $i \in I$, $p^{n_i}$ is an actual isogeny for some integer $n_i$.
    
    Take the product $\cG := \prod_{i \in I} \cG_i$, which by Corollary \ref{prod_cor} is a $p$-divisible group with $G$-structure over $R^{\sharp_2+}$. Also there is a unique trivialization $\psi: T_p\cG \iso \underline{\Lambda}$ that restricts to the $\psi_i$, by properness of the diagonal of $[*/\underline{K}_{hs}] \to *$, see \cite[2.18]{Zha23}. Let $\cG_0 := \cG \times_{R^{\sharp_2+}} R^+/\varpi$.

    Next, we recall from \cite[8.14]{Zha23} how to show that the integers $n_i$ can be bounded by a sufficiently large integer $N \gg 0$. This amounts to showing that there is a common bound $N \gg 0$ such that $p^N\phi_i$ are all actual isogenies. First, $(\cG,\psi)$ defines an $S$-valued point in $\Gr$, by considering its prismatic Dieudonné module $M_\Prism(\cG)$, tensoring it up to $B_{\dR}^+(R^{\sharp_2})$ and comparing it with
    \[T_p\cG \otimes_{\Z_p} B_{\dR}^+(R^{\sharp_2}) \overset{\psi \otimes \id}{\simeq} \Lambda \otimes_{\Z_p} B_{\dR}^+(R^{\sharp_2}),\]
    which yields a $G$-torsor on $\Spec(B_{\dR}^+(R^{\sharp_2}))$ which is trivialized over $B_{\dR}(R^{\sharp_2})$ by crystalline-de Rham comparison.
    This $S$-point agrees with our original point $y \in \Gr(S)$ when restricted to each $s_i$. Thus by properness of $\Gr$, the two $S$-points agree as a whole. In particular, the $G$-bundles $\ccE(\cG) := \ccE(\cG_0)$ and $\ccE(y)$ are isomorphic, and thus the isomorphism $\phi: \ccE(\cA[p^\infty]) \overset{\sim}{\to} \ccE(y)$ can be seen instead as an isomorphism $\phi: \ccE(\cA[p^\infty]) \overset{\sim}{\to} \ccE(\cG)$. Letting $(M,\varphi_M)$ resp. $(M',\varphi_{M'})$ denote the crystalline Dieudonné module of $\cA$ resp. $\cG$, we have a diagram as follows
\[\begin{tikzcd}
	{(M^\vee \otimes M')^{\varphi \otimes \varphi'=1}} && {(M^\vee \otimes M'[\frac{1}{p}])^{\varphi \otimes \varphi' = 1}} \\
	{\Hom(\cA[p^\infty],\cG_0)} && {\Hom(\ccE(\cA[p^\infty]),\ccE(\cG))} \\
	{\prod_i \Hom(\cA_{i}[p^\infty],\cG_{0,i})[\frac{1}{p}]} && {\prod_i \Hom(\ccE(\cA_{i}[p^\infty]),\ccE(\cG_i))}
	\arrow[hook, from=1-1, to=1-3]
	\arrow["\sim", from=1-1, to=2-1]
	\arrow["\sim", from=1-3, to=2-3]
	\arrow["{\ccE(\cdot)}", from=2-1, to=2-3]
	\arrow[hook', from=2-1, to=3-1]
	\arrow[hook', from=2-3, to=3-3]
	\arrow["{\ccE(\cdot) \sim}", from=3-1, to=3-3]
\end{tikzcd}\]
    This makes it clear that for large enough $N \gg 0$, $p^N\phi \in \Hom(\ccE(\cA[p^\infty]),\ccE(\cG))$ lies in the image of $\ccE(\cdot)$. Hence $p^N(\phi_i)_{i \in I}$ also lies in the image of $\ccE(\cdot)$. This $N$ serves as the desired upper bound. 
 
    By taking products, we get an $F_p$-quasi-isogeny
    \[\rho = p^{-N} \prod_i p^N\rho_i: \cA[p^\infty] \longrightarrow \cG_0 \]
    such that $\ccE(\rho)$ gets identified with $\phi$.
    
    Using Lemma \ref{8.11 adaptation} (or rather the remark following it), we can modify $\cA$ via $\rho$: It yields a quadruple $\cA' = (A',\iota',\lambda',\overline{\eta}')$ (unique up to isomorphism in $\M(R^+/\varpi)$), such that the induced map $\rho': \cA_0'[p^\infty] \to \cG_0$ is an $\O_p$-isomorphism. Furthermore, the quotient map $\pi: \cA \to \cA'$ will define an isomorphism in $\Igs$ once we have shown that $\cA'$ can be lifted to a formal abelian scheme over $R^{\sharp_2+}$.
    
    In order to lift, modify the polarization of $\cG$ by a scalar in $\underline{\O}_p^\times(S)$ such that $\rho'$ preserves the polarizations on the nose. Thus by the Serre-Tate Theorem \ref{st_2}, the triple $(\cA',\cG,\rho')$ gives rise to a formal abelian scheme $\fA'$ over $\Spf R^{\sharp_2+}$, equipped with $G$-structure and an infinite away-from-$p$ level (because we have only modified $\cA$ at $p$). Also, its $p$-divisible group $\fA'[p^\infty]$ agrees with $\cG$, hence $\psi: T_p\fA'_\eta = T_p\cG \iso \underline{\Lambda}$ gives the infinite level at $p$.
\end{proof}

\begin{corollary}\label{cart_hybrid}
    The diagram
\[\begin{tikzcd}
	{\cX^{\circ}} & \Gr \\
	\Igs & {\Bun_G}
	\arrow["{\pi_{HT}}", from=1-1, to=1-2]
	\arrow["q"', from=1-1, to=2-1]
	\arrow["BL", from=1-2, to=2-2]
	\arrow["{\overline{\pi}_{HT}}"', from=2-1, to=2-2]
\end{tikzcd}\]
is Cartesian.
\end{corollary}

\subsection{Abelian-type Igusa stack}
We now want to construct the Igusa stack for our abelian-type Shimura variety.

First, we note that there is a well-defined polarization action on the presheaf
\[S \mapsto \Igs^{\pre}(S)/\iso\]
on $\Perf$. Namely, given a point $(S^\sharp,\fA)$ of $\Igs^{\pre}(S)$, we can realize $\fA$ as an inverse limit of formal abelian schemes with $G$-structure and increasing tame level structure. At level $K_{hs}K^p$, the polarization action on the abelian scheme factors through the finite group $\Delta(K_{hs}K^p)$. Thus, in the limit we obtain a polarization action of the profinite group
\[\Delta_{hs} := \invlim_{K^p} \Delta(K_{hs}K^p),\]
and after v-sheafification we have an action by the same profinite group on the v-sheaf $\Igs$. By pulling back along the obvious quotient map $\Delta \twoheadrightarrow \Delta_{hs}$, we also obtain an action on $\Igs$ by the full profinite group $\Delta$. 

These considerations lead us to defining the (infinite level) Igusa stack for our abelian-type Shimura variety as follows.
\begin{definition}
    The \textbf{abelian-type Igusa stack} is the quotient stack in the category of v-stacks
    \[\Igs_G := [\Igs/\underline{\Delta}]\]
\end{definition}

The next two easy lemmas provide the necessary maps from and into $\Igs_G$.

\begin{lemma}
    The previously constructed map $\overline{\pi}_{HT} : \Igs \to \Bun_G$ factors through the polarization action by $\underline{\Delta}$.
\end{lemma}
\begin{proof}
    It suffices to check invariance of $\overline{\pi}_{HT}: \Igs^{\pre} \to \Bun_G$ with respect to the action of $\underline{\Delta_{hs}}$. But one checks immediately from the construction of $\overline{\pi}_{HT}$ that changing the polarization by totally positive units will change the symplectic pairing of the vector bundle $\ccE(A)$ by the same totally positive unit, but this does not alter the isomorphism class of the corresponding $G$-bundle.
\end{proof}

\begin{lemma}
    The map $q: \cX^{\circ} \to \Igs$ induces a map $q: \cX^{\circ}_G \to \Igs_G$.
\end{lemma}
\begin{proof}
    It suffices to see that $q: \cX^{\circ,\pre} \to \Igs^{\pre}$ is equivariant with respect to the action by $\Delta$, but this is obvious.
\end{proof}

With the hard work already been done, we can now establish the rational fiber product conjecture for our abelian-type Shimura variety.

\begin{corollary}\label{cart_abelian}
    The diagram
\[\begin{tikzcd}
	{\cX^{\circ}_G} & \Gr \\
	{\Igs_G} & {\Bun_G}
	\arrow["{\pi_{HT}}", from=1-1, to=1-2]
	\arrow["q"', from=1-1, to=2-1]
	\arrow["BL", from=1-2, to=2-2]
	\arrow["{\overline{\pi}_{HT}}"', from=2-1, to=2-2]
\end{tikzcd}\]
is Cartesian.
\end{corollary}
\begin{proof}
    Consider the composition of 2-commutative diagrams 
\[\begin{tikzcd}
	{\cX^\circ} & {\cX^{\circ}_G} & \Gr \\
	\Igs & {\Igs_G} & {\Bun_G}
	\arrow["\Delta", from=1-1, to=1-2]
	\arrow["q"', from=1-1, to=2-1]
	\arrow["{\pi_{HT}}", from=1-2, to=1-3]
	\arrow["q"', from=1-2, to=2-2]
	\arrow["BL", from=1-3, to=2-3]
	\arrow["\Delta"', from=2-1, to=2-2]
	\arrow["{\overline{\pi}_{HT}}"', from=2-2, to=2-3]
\end{tikzcd}\]
By Corollary \ref{cart_hybrid}, the full diagram is Cartesian, and the left-hand square consists of two maps that are v-torsors with respect to the same profinite group $\Delta$ (for the map $\cX^\circ \to \cX_G^\circ$, this is Lemma \ref{delta_torsor}). Hence also the right-hand square must be Cartesian. 
\end{proof}

Since the Hecke action of $G(\A_f)$ on $\Igs^{\pre}$ clearly commutes with the polarization action by $\Delta$, we obtain an induced Hecke action on $\Igs_G$. Furthermore, the Hecke action by $G(\A_f^p) \times G(\Q_p)$ on $\Igs^{\pre} \times_{\Bun_G} \Gr$ clearly recovers the Hecke action on $\cX^{\circ,\pre}$, thus after sheafifying the isomorphism
\[\cX^\circ \overset{\sim}{\to} \Igs \times_{\Bun_G} \Gr\]
is Hecke equivariant. This also shows that the isomorphism
\[\cX_G^\circ \overset{\sim}{\to} \Igs_G \times_{\Bun_G} \Gr\]
is Hecke equivariant. 

Hence the first axiom of Definition \ref{igs_axioms} is satisfied, which by the work of \cite{Kim25} immediately grants us uniqueness of the Igusa stack $\Igs_G$. In particular, this shows that our earlier choice of totally positive uniformizers $x_\P$ for each prime ideal $\P$ of $F$ above $p$ was in fact an auxiliary choice and our construction of $\Igs_G$ did a posteriori not depend on this choice. We remark that $\Igs$ as well as $\Igs_G$ also satisfy the second axiom from Definition \ref{igs_axioms}, by the exact same argument as \cite[5.2.15]{DHKZ24}.

\end{document}